\documentclass[12pt]{amsart}

\usepackage{graphicx,color}
\usepackage{setspace}
\usepackage{amsmath, amssymb,amsthm, amsfonts}

\usepackage{accents}

\newtheorem{thm}{Theorem}[section]
\newtheorem{lma}{Lemma}[section]
\newtheorem{prop}{Proposition}[section]
\newtheorem{cor}{Corollary}[section]

\theoremstyle{definition}
\newtheorem{definition}{Definition}[section]

\theoremstyle{remark}
\newtheorem{remark}{Remark}[section]

\numberwithin{equation}{section}

\newcommand{\tr}{\mbox{tr}}

\renewcommand{\div}{\mbox{div}}

\newcommand{\Ric}{\mbox{Ric}}
\newcommand{\R}{\mathbb R}

\newcommand{\be}{\begin{equation}}
\newcommand{\ee}{\end{equation}}
\newcommand{\bee}{\begin{equation*}}
\newcommand{\eee}{\end{equation*}}

\def\ppt{\frac{\p}{\p t}}

\def\p{\partial}

\def\la{\langle}
\def\ra{\rangle}
\def\lf{\left}
\def\ri{\right}
\def\Pi{\displaystyle{\mathbb{II}}}
\def\Rm{\text{Rm}}

\def\m{\mathfrak{m}}

\def\e{\epsilon}

\def\a{\alpha}

\def\wt{\widetilde }

\def\wn{\wt\nabla}

\def\ring{\accentset{\circ}}

\begin{document}

\title[]
{Scalar curvature and singular metrics}

 \author{Yuguang Shi$^1$}
\address[Yuguang Shi]{Key Laboratory of Pure and Applied Mathematics, School of Mathematical Sciences, Peking University, Beijing, 100871, P.\ R.\ China}
\email{ygshi@math.pku.edu.cn}
\thanks{$^1$Research partially supported by NSFC 11671015}

\author{Luen-Fai Tam$^2$}
\address[Luen-Fai Tam]{The Institute of Mathematical Sciences and Department of
 Mathematics, The Chinese University of Hong Kong, Shatin, Hong Kong, China.}
 \email{lftam@math.cuhk.edu.hk}
\thanks{$^2$Research partially supported by Hong Kong RGC General Research Fund \#CUHK 14305114}

\renewcommand{\subjclassname}{
  \textup{2010} Mathematics Subject Classification}

\date{November, 2016}

\begin{abstract}
Let $M^n$,  $n\ge3$, be a compact differentiable manifold with nonpositive Yamabe invariant $\sigma(M)$. Suppose $g_0$ is a continuous metric  with volume $V(M,g_0)=1$,   smooth outside a compact set $\Sigma$, and is in $W^{1,p}_{loc}$ for some $p>n$. Suppose the scalar curvature of $g_0$ is at least $\sigma(M)$ outside $\Sigma$.  We prove that $g_0$ is Einstein outside $\Sigma$  if  the codimension of $\Sigma$ is at least $2$.   If in addition, $g_0$ is Lipschitz then $g_0$ is smooth and Einstein after a change of the smooth structure.  If $\Sigma$  is a compact  embedded hypersurface,  $g_0$ is smooth up to $\Sigma$ from two sides of $\Sigma$, and if  the difference of the mean curvatures along $\Sigma$  at two sides of  $\Sigma$ has a fixed appropriate sign,  then $g_0$ is also Einstein outside $\Sigma$. For     manifolds with dimension between $3$ and $7$,  without spin assumption  we   obtain a positive mass theorem on an asymptotically flat manifold for metrics with a compact singular set of codimension at least $2$.
\end{abstract}


\maketitle

\markboth{Yuguang Shi and Luen-Fai Tam}{Scalar curvature and singular metrics }

\section{introduction}\label{s-intro}

There are two celebrated results on manifolds with nonnegative scalar curvature. The first result is on compact manifolds. It was proved by Schoen and Yau \cite{SchoenYau1979,SchoenYau1979-1} that any smooth metric on a torus $T^n$, $n\le  7$ with nonnegative scalar curvature must be flat. Later, the result was  proved to be true for all $n$  by Gromov and Lawson \cite{GromovLawson}. The second result is the positive mass theorem on noncompact  manifolds.   Schoen and Yau \cite{SY1979,SY1981,Schoen1989} proved that the  Arnowitt-Deser-Misner (ADM) mass of each end of an $n$-dimensional  asymptotically flat (AF) manifold with $3\le n\le 7$ with nonnegative scalar curvature is positive and if the  ADM  mass of an end is zero, then the manifold is isometric to the Euclidean space. Under the additional assumption that the manifold is spin,   the same result is still true and was proved by Witten \cite{Witten}, see also \cite{ParkerTaubes,BTK86}. In the two results the metrics are assumed to be smooth.

There are many results on positive mass theorem for nonsmooth metrics.   Miao \cite{Miao2002} and the authors \cite{ShiTam2002} studied and proved positive mass theorems for metrics with corners. The metrics are smooth away from a compact hypersurface, which are Lipschitz and satisfy certain conditions on the mean curvatures of the hypersurface. The result was used to prove the positivity of the Brown-York quasilocal mass   \cite{ShiTam2002}. In \cite{Lee2013}, Lee considered positive mass theorem for metrics with bounded $C^2$ norm and are smooth away from a singular set with codimension greater than $n/2$, where $n$ is the dimension of the manifold. On the other hand, McFeron and Sz\'ekelyhidi \cite{McFeronSzekelyhidi2012} were able to prove Miao's result using Ricci flow and Ricci-DeTurck flow, which was studied in details   by Simon \cite{Simon2002}. More recently, Lee and Feloch \cite{LeeFeloch2015} are able to prove  for spin manifolds, under rather general conditions, a positive mass theorem for metrics which may be singular. Their theorem can be applied to all previous results for nonsmooth metrics under the additional  assumption that the manifold is spin.

Motivated by these study of singular metrics on AF manifolds, we want to understand singular metrics on compact manifolds. One of the question is to see  if   there  are nonflat metrics with nonnegative scalar curvature on $T^n$ which may be singular somewhere. Another question  can be described as follows. It is now well-known that in every conformal class of smooth metrics on a compact manifold without boundary there is a metric with constant scalar curvature  by the works of  Yamabe, Trudinger, Aubin and Schoen, see \cite{Yamabe1960,Trudinger1968,Aubin1976-1,Aubin1976,Schoen1984}. One motivation for the result is to obtain Einstein metric. It is well-known that if a smooth metric on a compact manifold attains the Yamabe invariant and if the invariant is nonpositive, then the metric is Einstein. See \cite[p.126-127]{Schoen1989}. In this work, we will study the question whether this last result is still true for nonsmooth metrics.

 Let us recall the definition of Yamabe invariant, which is called $\sigma$-invariant in \cite{Schoen1989}.  Let $\mathcal{C}$ be a conformal class of smooth Riemannian metrics $g$, then the {\it Yamabe constant of $\mathcal{C}$} is defined as:
$$
Y(\mathcal{C})=\inf_{g\in\mathcal{C}}\displaystyle{\frac{\int_M\mathcal{S}_gdv_g}{(V(M,g))^{1-\frac2n}}}.
$$
where $\mathcal{S}_g$ is the scalar curvature and $V(M,g)$ is the volume of $M$ with respect to $g$. The {\it Yamabe invariant}  is defined as
$$
\sigma(M)=\sup_{\mathcal{C}}Y(\mathcal{C}).
$$
The supremum is taken among all conformal classes of smooth metrics. It is finite, see \cite{Aubin1976}.  If $\sigma(M)>0$, then in general it is still unclear whether $g$ is Einstein or not.

 To answer the question on Einstein metrics,  let $M^n$ be a compact smooth manifold without boundary and let $g_0$ be a continuous Riemannian metric on $M$  with $V(m,g_0)=1$ such that it is smooth outside a compact set $\Sigma$. The first case is that $\Sigma$ has codimension at least 2 and $g_0$ is in $W_{\text{loc}}^{1,p}$ for some $p>n$ (see  sections \ref{s-gradient} and \ref{s-Einstein-1} for more precise definitions).

\begin{thm}\label{t-intro-1} Let $(M^n,g_0)$ be as above. Suppose $\sigma(M)\le0$ and suppose the scalar curvature of $g_0$ outside $\Sigma$ is at least $\sigma(M)$. Then $g_0$ is Einstein outside $\Sigma$. If in addition that $g_0$ is  Lipschitz, then after changing the smooth structure, $g_0$ is smooth and Einstein.
\end{thm}

In case $\Sigma$ is a compact embedded hypersurface, as in \cite{Miao2002} we assume that near $\Sigma$,   $g_0=dt^2+g_{\pm}(z,t)$, $z\in \Sigma$ so that $(t,z)$ are smooth coordinates and $g_-(\cdot,0)=g_+(\cdot,0)$, where $g_+$, $g_-$ are smooth up to $\Sigma$. Moreover,    with respect to the unit normal $\frac{\p}{\p t}$ the mean curvature $H_+$  of $\Sigma$ with respect to $g_+$ and the mean curvature $H_-$  of $\Sigma$ with respect to $g_-$ satisfies $H_-\ge H_+$. Under these assumptions, we have:

\begin{thm}\label{t-intro-2} Let $(M^n,g_0)$ be as above  with $V(m,g_0)=1$. Suppose $\sigma(M)\le0$ and suppose the scalar curvature of $g_0$ outside $\Sigma$ is at least $\sigma(M)$. Then $g_0$ is Einstein outside $\Sigma$. Moreover, $H_+=H_-$.
\end{thm}

Note that it is easy to construct examples so that the theorem is not true if the assumption $H_-\ge H_+$ is removed.

In the process of proofs, one also obtain the following: In case $M^n$ is $T^n$,  under the regularity assumptions in Theorem \ref{t-intro-1} or Theorem \ref{t-intro-2} and if $g_0$ has nonnegative scalar curvature outside $\Sigma$, then $g_0$ is flat outside $\Sigma$.

The method of proofs of the above results can also be adapted to AF manifolds.   We  want to discuss positive mass theorem with singular metric on AF manifold with dimension $3\le n\le 7$ without assuming that the manifold is   spin. We will prove the following:

\begin{thm}\label{t-intro-3}
Let $(M^n,g_0)$ be an AF  manifold  with $3\leq n\leq 7$,  $g_0$ be a continuous metric on $M$ with regularity assumptions as in Theorem \ref{t-intro-1}. Suppose $g_0$ has nonnegative scalar curvature outside $\Sigma$. Then the ADM mass of each end is nonnegative. Moreover, if the ADM mass of one of the ends is zero, then $M$ is diffeomorphic to $\R^n$ and is flat outside $\Sigma.$
\end{thm}

We should mention that all the results mentioned above for nonsmooth metrics, all the metrics are assumed to be continuous.  On the other hand, one can construct an example  of AF metric  with a cone singularity and   nonnegative scalar curvature and with negative ADM mass, see section \ref{s-examples}.
One can also construct examples of metrics on compact manifolds with a cone singularity so that Theorem \ref{t-intro-1} is not true. In these examples, the metrics are not continuous.

The structure of the paper is as follows: in section \ref{s-examples}, we construct examples which are related to results in later sections; in section \ref{s-gradient} we obtain some estimates for the Ricci-DeTurck flow; in section \ref{s-approx-1} we use the Ricci-DeTurck flow to approximate   singular metrics;  in sections \ref{s-Einstein-1} and   \ref{s-harmonic-1} we will prove Theorems \ref{t-intro-1} and \ref{t-intro-2}; in section \ref{s-pmt} we will prove Theorem \ref{t-intro-3}.  In this work, the dimension of any manifold is assumed to be at least three. We   will also use the Einstein summation convention.

The authors would like to thank Xue Hu and Richard Schoen for some useful discussions.

\section{examples of metrics with cone singularities}\label{s-examples}

In   previous results on positive mass theorems on AF manifolds with singular metrics mentioned in section \ref{s-intro},  the metrics are all assumed to be continuous. To understand this condition on continuity and to motivate our study, in this section,  we   construct some examples  with cone singularities  which are related to the study in the later sections.

 The following lemma is standard. See \cite{Petersen1998}.

\begin{lma}\label{l-warped-1}
Consider the metric $g=dr^2+\phi^2(r) h_0$ on $(0,r_0)\times \mathbb{S}^{n-1}$, where $h_0$ is the standard metric of $\mathbb{S}^{n-1}$, $n\ge 3$ and $\phi$ is a smooth positive function  on $(0,r_0)$. Then the scalar curvature of $g$ is given by
$$
\mathcal{S}=(n-1)\lf[- \frac{2\phi''}{\phi}+ (n-2)\frac{1-(\phi')^2}{\phi^2}\ri].
$$
 Suppose $\phi=\a r^\beta$, with $\a, \beta>0$. Then
$\mathcal{S}>0$    if $\a<1$, $\beta=1$ or if $0<\beta\le \frac2n$. In both cases, the metric is not continuous up to $r=0$. If $\a>1, \beta=1$, then $\mathcal{S}<0$ for $r$ small enough.
\end{lma}

We can construct asymptotically flat manifold with nonnegative scalar curvature defined on $\R^3\setminus\{0\}$ such that the metric behaves like $dr^2+(\a r)^2 h_0$ for some $0<\a<1$  with positive mass.
\begin{prop}\label{p-cone-pm-1}
  Let  $0<\e<\frac12$ and let $\eta(x)=\eta(r)$, with $r=|x|$, be a smooth function on $\R^3\setminus \{0\}$ such that
\bee
\left\{
  \begin{array}{ll}
    \eta(r)=-\e(1-\e)r^{-\e-2}, & \hbox{if $0<r\le 1$;} \\
    \eta(r)<0, & \hbox{if $1\le r\le 2$;}\\
\eta(r)=0, &\hbox{if $r\ge 2$.}
  \end{array}
\right.
\eee
Let
$\phi$ be the function defined on $\R^3\setminus\{0\}$ with
$$
\phi(r)=\int_1^r\frac1{s^2}\lf(\int_0^st^2\eta(t)dt\ri)ds.
$$
Then there are constants $a, b>0$ such that if
$$
u=\phi+b+\frac a2+1
$$
 then $u>0$. Moreover, if $g=u^4g_\e$ where $g_e$ is the standard Euclidean metric, then near infinity,
$$
g=(1+\frac a r)^4 g_e,
$$
and near $r=0$,
$$
g=d\rho^2+\lf((1-2\e)^2\rho^2+O(\rho^{2+\delta})\ri)h_0
$$
for some $\delta>0$, where
$$
\rho=\int_0^{r} u^2(t) dt.
$$
$g$ has nonnegative  scalar curvature and has zero scalar curvature outside a compact set. Moreover, the end near infinity   is asymptotically flat in the sense of Definition \ref{defaf} in section \ref{s-pmt},  and has positive mass $2a$.
\end{prop}
\begin{proof} Let $\Delta_0$ be the Euclidean Laplacian. Then one can check that
$$
\Delta_0\phi=\eta\le0.
$$
For $0< r\le 1$,
$$
\phi(r)=r^{-\e}-1.
$$
For $r\ge 2$, let
$$
a=-\int_0^r s^2\eta(s)ds>0,
$$
and
$$
b=-\int_1^2\frac1{s^2}\lf(\int_0^s \tau^2\eta(\tau)d\tau\ri)ds>0.
$$
Then
\bee
\begin{split}
\phi(r)=&-b+\int_2^r\frac1{s^2}\lf(\int_0^st^2\eta(t)dt\ri)ds\\
=&-b-a\int_2^r\frac1{s^2}ds\\
=&-b-\frac a2+\frac ar.
\end{split}
\eee
Hence if $u=\phi+b+\frac a2+1$, then $\Delta_0 u=\eta\le0$.   Since $u\to\infty$ as $r\to 0$ and  $u\to 1$ as $r\to \infty$, $u>0$ by the strong maximum principle. The metric
$$
g=u^4g_e
$$
is defined on $\R^3\setminus\{0\}$, has nonnegative scalar curvature and has zero scalar curvature near infinity. $g$ is also asymptotically flat.
Near $r=0$,
$$
u=b+\frac a2+r^{-\e}.
$$

Since  $0<\e<\frac12$, we let $$
\rho=\int_0^r u^2(t) dt=\frac1{(1-2\e)}r^{1-2\e}+O(r^{1-\e}).
$$
So
$$
\rho^2= \frac1{(1-2\e)^2}r^{2-4\e}+O(r^{2-3\e}).
$$
Hence near $r=0$,
\bee
\begin{split}
g=&d\rho^2+u^4r^2h_0 \\
=& d\rho^2+(r^{2-4\e}+O(r^{2-3\e}))h_0\\
=&d\rho^2+((1-2\e)^2\rho^2+O(r^{2-3\e}))h_0\\
=&d\rho^2+(\a^2\rho^2+O(r^{2-3\e}))h_0
\end{split}
\eee
where $\a=1-2\e$.
Note that $r^{2-3\e}=O(\rho^{2+\delta})$ for some $\delta>0$.

\end{proof}

The following example is the type of singularity which is called zero area singularity in
\cite{BrayJauregui2013}.

\begin{prop}\label{p-cone-pm-2} Let $m>0$ and let $\phi=1-\frac {2m}r$. Then the metric
$$
g=\phi^4g_e
$$
is asymptotically flat defined on $r>2m$ in $\R^3$, with zero scalar curvature and with negative mass $-m$. Moreover, near $r=2m$,
$$
g=d\rho^2+c\rho^\frac43(1+O(\rho^\frac23))h_0
$$
for some $c>0$,
where
$$
\rho=\int_0^{r-2m}\phi^2(t+2m)dt.
$$
Hence near $\rho=0$ the metric is asymptotically of the form as in  Lemma \ref{l-warped-1} with $\beta=\frac23$.
\end{prop}

\begin{proof}
We only need to consider $g$ near $r=2m$. The rest is well-known. Let $t=r-2m$, $r>2m$. Then
$$
\wt \phi(t)=\phi(t+2m)=\frac{t}{t+2m}=\frac{t}{2m}\lf(1-\frac t{2m}+\frac{t^2}{4m^2}+O(t^3)\ri).
$$
$$
\rho=\int_0^t\wt \phi^2(s)ds=\int_0^t\frac{s^2}{ (s+2m)^2}ds.
$$
Note that as $r\to 2m$, $\rho\to0$. In terms of $\rho$, near $\rho=0$,
$$
g=d\rho^2+\phi^4r^2h_0.
$$
Near $\rho=0$,
\bee
\begin{split}
\phi^4r^2=&\frac{t^4}{(t+2m)^4}(t+2m)^2\\
=&c\rho^\frac43(1+O(\rho^\frac23))
\end{split}
\eee
for some $c>0$.
\end{proof}

We can also construct conical metric on $T^3\setminus \{ \text {a point}\}$, with nonnegative scalar curvature and with positive scalar curvature somewhere.

First, we have
\begin{prop}\label{p-gluing1} Let $m>0$
There is a metric $g$ on $\mathbf{R}^3\setminus B(2m)$ satisfying:
 \begin{enumerate}
   \item[(i)] The scalar curvature  $R\geq 0$ and $R>0$ somewhere;
   \item[(ii)] There exist $r_0$ and $r_1$ with $r_1>r_0>2m$ so that $g=(1-\frac{2m}r)^4 g_e$ for any $r\in (2m, r_0)$ and $g=g_e$ for any $r\geq r_1$, where $g_e$ is the Euclidean metric.

 \end{enumerate}
 \end{prop}

\begin{proof} Let $r_1>r_0>2m$ to be chosen later.
Let $\eta(r)$ be a  smooth nonincreasing function with
\begin{equation}
\eta(r)=\left\{
       \begin{array}{ll}
         2m, & \hbox{$2m\leq r\leq r_0$;} \\
         0, & \hbox{$r\geq r_1$.}
       \end{array}
     \right.
\end{equation}
For any $\rho\geq 2m$, let

$$
y(\rho)=\int^\rho _{2m} \frac{\eta(r)}{r^2} dr,
$$
By choosing suitable   $r_0, r_1$, we may get $y(\rho)=1$ for any $\rho\geq r_1$,
then we see that

\begin{equation}
y(r)=\left\{
       \begin{array}{ll}
         1-\frac{2m}{r}, & \hbox{$2m\leq r\leq r_0$;} \\
         1, & \hbox{$r\geq r_1$.}
       \end{array}
     \right.
\end{equation}
We claim that

$$
\Delta_0 y \leq 0, \quad \text{on $\R^3\setminus B_{2m}$},
$$
here $\Delta_0$ is the standard Laplace operator on $\R^3$. By a direct computation, we see that

\begin{equation}
\begin{split}
\Delta_0 y =y''+\frac2r y'=r^{-2}(r^2 y')'=r^{-2}\eta'\leq 0
\end{split}
\end{equation}

For any $x\in\R^3\setminus B_{2m} $, let $u(x)=y( |x|)$, then $g=u^4 (dr^2 +r^2 h_0)$ is the required metric.
\end{proof}

Suppose  $T^3 (r)$ is flat torus, by taking $r$ large enough we may glue $(B_r\setminus B_{2m},g)$ with $T^3 (r) \setminus B_r$ directly. As in Proposition \ref{p-cone-pm-2}, near $r=2m$, the metric can be considered as a metric with cone singularity. The question is whether we have a metric on $n$-torus which has a cone singularity of the form $dr^2+\a^2 r^2h_0$ with $0<\a<1$ and with nonnegative scalar curvature. This will be answered in section \ref{s-approx-1}.  The problem can be reduced to the study of singular metrics on $T^n$ with nonnegative scalar curvature.

\section{gradient estimates for solution to the $h$-flow}\label{s-gradient}

We want to use the    Ricci-DeTurck flow to deform a singular metric to a smooth one. We need some basic facts about the flow.

Let $(M^n, h)$ be a complete manifold without boundary. We assume that the curvature of $h$ and   its covariant derivatives are bounded:
\be\label{e-h-bounds-1}
 |\wn^{(i)} \wt Rm |\le k_i
 \ee
 for all $3\ge i\ge0$. Here $\wn$ is the covariant derivative with respect to $h$ and $\wt \Rm$ is the curvature tensor of $h$. A smooth  family of metrics $g(t)$  on $M\times(0,T]$, $T>0$,  is said to be a solution to the $h$-flow if $g(t)$ satisfies:
\be\label{e-hflow}\begin{split}
 \ppt g_{ij}=&g^{\a\beta}\wn_\a\wn_\beta g_{ij}-g^{\a\beta}g_{ip}h^{pq}\wt\Rm_{j\a q\beta}-g^{\a\beta}g_{jp}h^{pq}\wt\Rm_{i\a q\beta}\\
 &+\frac{1}{2}g^{\a\beta}g^{pq}(\wn_i g_{p\a}\cdot\wn_j g_{q\beta}+2\wn_\a g_{jp}\cdot\wn_q g_{i\beta}-2\wn_\a g_{jp}\cdot\wn_\beta g_{iq}\\
 &-2\wn_j g_{\a p}\cdot\wn_\beta g_{iq}-2\wn_i g_{\a p}\cdot\wn_\beta g_{jq}), \end{split} \ee
Let
\be\label{e-box}
\Box=\ppt-g^{ij}\wn_i\wn_j.
\ee
 For a constant $\delta>1$, $h$ is said to be   $\delta$ close to a metric $g$ if
 $$
 \delta^{-1}h\le g\le \delta h.
 $$
 In \cite{Simon2002}, Simon obtained the following:
  \begin{thm}\label{t-Simon} {\rm (Simon)}
  There exists $\e=\e(n)>0$ depending only on   $n$ such that if $(M^n,g_0)$ is an $n$-dimensional  compact or noncompact manifold  without boundary with continuous Riemmannian metric $g_0$ which is  $(1+\e(n))$ close to    a smooth  complete  Riemannian metric $h$ with curvature bounded by $k_0$,   then the $h$-flow \eqref{e-hflow} has a  smooth  solution on $M\times(0,T]$ for some $T>0$ with $T$ depending only on $n, k_0$ such that $g(t)\to g_0$ as $t\to 0$ uniformly on compact sets and such that
  $$
  \sup_{x\in M}|\wn^i g(t)|^2\le \frac{C_i}{t^i}
  $$
  for all $i$, where $C_i$ depends only on $n, k_0,\dots,k_i$ where $k_j$ is the bound of $|\wn^j \Rm(h)|$. Moreover, $h$ is $(1+2\e)$ close to $g(t)$ for all $t$. Here and in the following    $|\cdot|$ is the norm  with respect to $h$.
  \end{thm}

In case $g_0$ is smooth, and if $|\wn g_0|$ is bounded, then it is also proved in \cite{Simon2002} that
$$
|\wn g(t)|\le C;\ \  |\wn^2 g(t)|\le Ct^{-\frac12}.
$$
We want to obtain estimates in case $g_0\in W_{\text{loc}}^{1,p}$ in the sense that $|\wn g_0|$ is in $L^p_{\text{loc}}$, for $p>n$. We have the following:

\begin{lma}\label{l-gradient-1} Fix $p\ge 2$. There is $b=b(n,p)>0$ depending only on $n, p$, with $e^b\le 1+\e(n)$ where $\e(n)$ is the constant in Theorem \ref{t-Simon}, such that if $g_0$ is smooth metric which is $e^b$ close to $ h$,  where $h$ is smooth and satisfies   \eqref{e-h-bounds-1}  for $0\le i\le 2$, then solution $g(t)$ of the $h$-flow with initial metric $g_0$ on $M\times[0,T]$ described in Theorem \ref{t-Simon} satisfies the following estimates: There is a constant $C>0$  depending only $n, p, h$ such that  for any $x_0\in M$ with injectivity radius $\iota(x_0)$ with respect to $h$, the following estimate is true:
$$
|\wn g(t,x_0)|^2\le \frac{CD}{  t^{\frac n{2 p}}}
 $$
for $T>t>0$ where $D$ depends only $n$, the lower bound of $\iota(x_0)$ and the  $L^{2p}$ norm of $|\wn g_0|$ in $B(x_0,\iota(x_0))$ which is the geodesic ball with respect to $h$.
\end{lma}

 \begin{proof} Suppose $g_0$ is  $e^b<1+\e(n)$ close to $h$, then for any     $\lambda>0$, $\lambda g_0$ is also $e^b$ close to $\lambda h$. Moreover, if $g(t)$ is the solution to the $h$-flow, then $\lambda g(\frac1\lambda t)$ is a solution to the $\lambda h$-flow. Hence by  scaling, we may assume that  $k_0+k_1+k_2\le 1$. The solution $g(t)$ constructed in \cite{Simon2002} is $e^{2b}$ close to $h$. Moreover, we may assume that $T\le 1$.

  Denote $\iota(x_0)$ by $\iota_0$ and we may assume that $\iota_0\le 1$. In the following $c_i$ will denote a constant depending only on $n$. Let $m\ge 2$ be an integer, which will be chosen depending only on $n, p$. Let $b=\frac1{2m}$. First choose $m$ so that $e^b\le 1+\e(n)$. Let $f_1=|\wn g|$ and $\psi=\lf(a+\sum_{i=1}^n\lambda_i^m\ri)f_1^2$  with $a>0$, where     $\lambda_i$ are the eigenvalues of $g(t)$ with respect to $h$. By choosing $a$ depending only on $n$ and $m$ large enough depending only on $n$,   as in \cite{Shi1989,Simon2002}, see also \cite[(5.8)]{HuangTam2015}, we have
  \be\label{e-psi-1}
  \Box \psi\le c_1-c_2m^2 f_1^4
  \ee
  Let $x^i$ be normal coordinates in $B(x_0,\iota_0)$. Since $k_0+k_1+k_2\le 1$,  by \cite[Corollary 4.11]{Hamilton1995}, then on  $B(x_0,\iota_0)$ we have

 \be\label{e-Hamilton-1}
 \left\{
   \begin{array}{ll}
     \frac12 |\xi|^2\le  h_{ij}\xi^i\xi^j\le 2|\xi|^2,\ \ \text{\rm for $\xi\in\R^n$};   \\
     \lf|D^\beta_x  h_{ij}\ri|\le c_3,  \ \ \text{\rm for all $i, j$ },
   \end{array}
 \right.
 \ee
where $h_{ij}=h(\frac{\p}{\p x^i},\frac{\p}{\p x^j})$ and $\beta=(\beta_1,\dots,\beta_n)$ is a multi-index with $|\beta|\le 2$ and  $D_{x^k}=\frac{\p}{\p x^k}$.     Let $\eta$ be a smooth function on $[0,1]$ such that $0\le \eta\le 1$, $\eta(s)=0$ for $s\ge \frac34$, $\eta(s)=1$ for $0\le s\le \frac12$. Still denote $\eta(|x|/\iota_0)$ by $\eta(x)$.
Then $|\wn \eta|\le c_4\iota_0^{-1}$. We have
\bee
\begin{split}
\frac{d}{dt}\int_{B(x_0,\iota_0)}& \eta^2\psi^p dv_h\\
=&p\int_{B(x_0,\iota_0)} \eta^2\psi^{p-1}\psi_t dv_h\\
\le & p\int_{B(x_0,\iota_0)} \eta^2\psi^{p-1}g^{ij}\wn_i\wn_j \psi dv_h+ p\int_{B(x_0,\iota_0)}  \eta^2\psi^{p-1} (c_1-c_2m^2f_1^4) dv_h\\
\le & -pc_5\int_{B(x_0,\iota_0)}(p-1)\eta^2\psi^{p-2}|\wn\psi|^2 dv_h
+pc_6\int_{B(x_0,\iota_0)} \eta^2\psi^{p-1}f_1|\wn\psi| dv_h\\
&+pc_7\iota_0^{-1}\int_{B(x_0,\iota_0)}  \eta\eta'\psi^{p-1}|\wn\psi| dv_h
+ p\int_{B(x_0,\iota_0)}\eta^2\psi^{p-1} (c_1-c_2m^2f_1^4) dv_h\\
\le & \frac{4c_6p}{p-1}\int_{B(x_0,\iota_0)} f_1^2\eta^2\psi^p dv_h+ \frac{4c_7p}{(p-1)\iota_0^2}\int_{B(x_0,\iota_0)} (\eta')^2 \psi^pdv_h+
\\&
+ p\int_{B(x_0,\iota_0)}\eta^2\psi^{p-1} (c_1-c_2m^2f_1^4) dv_h\\
\le &\frac{c_8p}{p-1}\int_{B(x_0,\iota_0)} f_1^4\eta^2\psi^{p-1} dv_h+ \frac{4c_7p}{(p-1)\iota_0^2}\int_{B(x_0,\iota_0)} (\eta')^2 \psi^pdv_h+
\\&
+ p\int_{B(x_0,\iota_0)}\eta^2\psi^{p-1} (c_1-c_2m^2f_1^4) dv_h
\end{split}
\eee
where we have used the fact that $\psi\le cf_1^2$ for some constant  $c$ depending only on $n$ by the fact that $2bm=1$ so that $\lambda_i^m\le 1$ for all $i$. Hence by choosing $m$ large enough depending only on $n, p$ and if $b=\frac1{2m}$,  we have

\bee
\begin{split}
\frac{d}{dt}\int_{B(x_0,\iota_0)}  \eta^2\psi^p dv_h\le c_9p\iota_0^{-2}\lf(\int_{B(x_0,\iota_0)} (\eta')^2 \psi^pdv_h
+ \int_{B(x_0,\iota_0)}\eta^2\psi^{p-1}   dv_h\ri).\\
\end{split}
\eee
By replacing $\eta$ by $\eta^q$ for $q\ge 1$, we may assume that $|\eta'|\le C\eta^{1-\frac1q}$, where $C$ depends only on $q$. Let $q=2p$, say, then we have
\bee
\begin{split}
&\frac{d}{dt} \int_{B(x_0,\iota_0)}   \eta^2\psi^p dv_h\\
\le&  C_1 \iota_0^{-2}\lf(\int_{B(x_0,\iota_0)} (\eta^2)^{ 1-\frac1{2p} } \psi^pdv_h+\int_{B(x_0,\iota_0)}\eta^2\psi^{p-1}   dv_h\ri) \\
\le &C_1\iota_0^{-2}\lf[\lf(\int_{B(x_0,\iota_0)}\eta^2\psi^p  dv_h\ri)^{1-\frac1{2p}}\lf(\int_{B(x_0,\iota_0)} \psi^p  dv_h\ri)^{\frac1{2p}}+\lf(\int_{B(x_0,\iota_0)}\eta^2\psi^p  dv_h\ri)^{1-\frac1p} \ri]\\
\le &C_2\iota_0^{-2}\lf[\lf(\int_{B(x_0,\iota_0)}\eta^2\psi^p  dv_h\ri)^{1-\frac1{2p}}t^{-\frac12}+\lf(\int_{B(x_0,\iota_0)}\eta^2\psi^p  dv_h\ri)^{1-\frac1p} \ri]
\end{split}
\eee
here and below upper case $C_i$ denote a positive constant  depending only on $n, p$ and $h$. Here we have used the estimates in Theorem \ref{t-Simon}. Let
$$
F=\int_{B(x_0,\iota_0)}   \eta^2\psi^p dv_h+1.
$$
Then we have
$$\frac{d}{dt}F\le C_3\iota_0^{-2}F^{1-\frac1{2p}}t^{-\frac12}.
$$
 Let $I=\int_{B(x_0,\iota_0)}|\wn g_0|^{ 2p} dv_h$. We conclude that
\bee
F(t)\le C_4\lf(I+\iota_0^{-2p}\ri),
\eee
or
\bee
\int_{B(x_0,\frac12\iota_0)}\psi^p dv_h\le C_5 \lf(I+\iota_0^{-2p}\ri).
\eee
Hence $0<t_0<T$,  by the mean value equality \cite[Theorem 7.21]{Lieberman}  applied to \eqref{e-psi-1} to $B(x_0,r)\times (t_0-r^2,t_0)$ with $r=\frac12\sqrt {t_0}$, we have
 \bee
 \psi^p(x_0,t_0)\le C_6 r^{-n}\lf(I+\iota_0^{-2p}+1\ri).
 \eee
 From this the result follows.
 \end{proof}

Assume $2p>n$ and let $\delta=n/(2p)$.  Let $b$ as in Lemma \ref{l-gradient-1}. Assume $h$ satisfies \eqref{e-h-bounds-1}, for $0\le i\le 2$. Then we have the following:

\begin{lma}\label{l-gradient-2} Let $x_0\in M$ and let $r_0>0$. Let
$$
I:=\int_{B(x_0,r_0)}|\wn g_0|^{2p}dv_h.
$$
Let $\iota$ be the infinmum of the injectivity radii $\iota(x)$, $x\in B(x_0,r_0)$. Then there is a constant $C$ depending only on $n, p, h,  r_0$, lower bound of $\iota$ and upper bound of $I$ such that
$$
|\wn^2g(x_0,t)|^2\le C t^{-1-\delta}.
$$
\end{lma}
\begin{proof} In the following, $C_i$ will denote a constant depending only on  the quantities mentioned in the lemma. By Lemma \ref{l-gradient-1}, we have
\be\label{e-gradient-1}
\sup_{x\in B(x_0,\frac {r_0}2)}|\wn g(x, t)|^2\le C_1 t^{-\delta}.
\ee
Let $f_i=|\wn^i g|$. As in \cite{Shi1989,Simon2002}, see also \cite[(5.11)]{HuangTam2015}, one can find $a>0$ depending only on the quantities mentioned in the lemma such that if
$\psi=(at^{-\delta}+f_1^2)f_2^2$, then

\be\label{e-ddg-1}
\begin{split}
\square\psi\le -\frac18 f_2^4+C_2 t^{-4\delta}
 \end{split}
\ee
on $B(x_0,\frac{r_0}2)\times(0,T]$. We may assume that $\iota(x_0)\le \frac {r_0}2$. Let $\eta$  be a cutoff function so that $(\eta')^2+|\eta''|\le c\eta$ for some absolute constant as in the proof of Lemma \ref{l-gradient-2}, let $F=t^{1+2\delta}\eta\psi.$ Since $g$ is smooth up to $t=0$, and $f_1^2\le C_1 t^{-\delta}$, we have $F(\cdot,0)=0$. If $F$ has a positive maximum, then there is $x_1\in B(x_0,\iota)$ and $T\ge t_1>0$ such that
$$
F(x_1,t_1)=\sup_{B(x_0,\iota)\times[0,T]}F.
$$
Hence at $(x_1,t_1)$, we have
\bee
\eta\wn_i\psi+\psi\wn_i\eta=0
\eee
and
\bee
\begin{split}
0\le &\Box F\\
=& t_1^{1+2\delta} \lf(\eta\Box\psi+\psi \Box \eta-2g^{ij}\wn_i \psi\wn_j\eta\ri)+(1+2\delta) t_1^{-1}F \\
\le &t_1^{1+2\delta}\lf[\eta\lf(-\frac18 f_2^4+C_2 t^{-4\delta}\ri)-\psi g^{ij}\wn_i\wn_j \eta+2g^{ij}\eta^{-1}\psi\wn_i\eta\wn_j\eta
\ri]+(1+2\delta) t_1^{-1}F\\
\le &t_1^{1+2\delta}\lf[\eta\lf(-\frac18 f_2^4+C_2 t^{-4\delta}\ri)+C_3\psi
\ri]+(1+2\delta) t_1^{-1}F
\end{split}
\eee
Multiply the inequality by $t_1^{1+2\delta}\eta(at^{-\delta}+f_1^2)^2=F\psi^{-1}(at^{-\delta}+f_1^2)$, we have
\bee
\begin{split}
0\le &-\frac18 F^2+C_3t_1^{1+\delta}(at^{-\delta}+f_1^2)F+(1+2\delta)t^{2\delta}(at^{-\delta}+f_1^2)^2F\\
\le&  -\frac18 F^2+C_4F.
\end{split}
\eee
Hence $F\le 8C_4.$ From this it is easy to see that the result follows.
\end{proof}


\section{approximation of singular metrics}\label{s-approx-1}

Let $(M^n,\mathfrak{b})$ be a smooth complete manifold of dimension $n$ without boundary. Let $g_0$ be a continuous Riemannian metric on $M$ satisfying the following:

\begin{enumerate}
  \item [({\bf a1})] There is a compact subset $\Sigma$ such that $g_0$ is smooth on $M\setminus \Sigma$.
  \item [({\bf a2})] $g_0$ is in $W_{\text{loc}}^{1,p}$ for some $p\ge1$ in the sense that $g_0$ has weak derivative and $|g_0|_\mathfrak{b}$, $|^\mathfrak{b}\nabla g_0|_\mathfrak{b}\in L^p_{\text{loc}}$ with respect to the metric $\mathfrak{b}$.
\end{enumerate}

We want to approximate $g_0$ by smooth metrics  with uniform bound on the $W^{1,p}$ norm locally.
As in \cite{Lee2013}, cover $\Sigma$ by finitely many precompact coordinate  patches $U_1,\dots,U_N$ and cover $M$ with
$U_1,\dots,U_N$ and $U_{N+1}$ so that $U_{N+1}$ is an open set with $U_{N+1}\cap \Sigma=\emptyset$. we may assume that there is   a partition of unity   $\psi_k$ with $ \text{\rm supp}(\psi_k)\subset U_k$. Since $g_0$ is continuous,  we may assume that $g_0$, $\mathfrak{b}$ and the Euclidean metric are equivalent in each $U_k$, $1\le k\le N$. For any $a>0$, let $\Sigma(a)=\{x\in M| d_{\mathfrak{b}}(x,\Sigma)<a\}$.
By \cite[Lemma 3.1]{Lee2013}, for each $1\le k\le N$, there is a smooth function $\e\ge\rho_k\ge0$ in $U_k$ such that for $\e>0$ small enough:
\be\label{e-cutoff-1}
\left\{
  \begin{array}{ll}
    \rho_k= \e &\hbox{$\Sigma(\e)\cap U_k$;} \\
    \rho_k=   0 &\hbox{$U_k\setminus \Sigma(2\e)$}; \\
    |\p\rho_k|\le     C; &  \\
    |\p^2\rho_k|\le   C\e^{-1};&  \\
  \end{array}
\right.
\ee
for some $C$ independent of $\e$ and $k$. Here $\p\rho_k$ and $\p^2\rho_k$ are   derivatives with respect to the Euclidean metric. Let $g^{k}_{0}=\psi_kg_0$ and for $1\le k\le N$, let
\be\label{e-approx-1}
(g^{k}_{\e,0})_{ij}(x)=\int_{\R^n}g_{0,ij}^k(x-\lambda\rho_k(x)y)\varphi(y)dy
\ee
Here $\varphi$ is a nonnegative smooth function in $\R^n$ with support in $B(1)$ and integral equal to 1. $\lambda>0$ is a constant independent of $\e$ and $k$, to be determined.
Finally, define
\be\label{e-approx-1}
g_{\e,0}=\sum_{k=1}^Ng^{k}_{\e,0}+\psi_{N+1}g_0.
\ee

\begin{lma}\label{l-approx-1} For $\e>0$ small enough, $g_{\e,0}$ is a smooth metric such that  $g_{\e,0}$ converge to $g_0$ in $C^0$ norm, $g_{\e,0}=g_0$ outside $\Sigma(2\e)$. Moreover, there is a constant $C$ independent of $\e$ such that
$$
\int_{\Sigma(1)}|^\mathfrak{b}\nabla g_{\e,0}|_\mathfrak{b}^pdv_{\mathfrak{b}}\le C.
$$
\end{lma}
\begin{proof} It is easy to see that $g_{\e,0}$ is smooth and converge to $g_0$ uniformly as $\e\to0$.  In order to estimate the $W^{1,p}_{\text{loc}}$ norm of $g_{\e,0}$, it is sufficient to estimate the norm in each $U_k$, $1\le k\le N$. Moreover, we may assume that $\mathfrak{b}$ is the Euclidean metric. So  it is sufficient to prove the following: For fixed $k$, $1\le k\le N$ and for any $u\in W^{1,p}_{\text{loc}}$ if
$$
v(x)=\int_{\R^n}u(x-\lambda\rho_k(x)y)\varphi(y)dy,
$$
then the $W^{1,p}$ norm of $v$ in $\Sigma(1)$ can be estimated in terms of the $W^{1,p}$ norm of $u$ in $\Sigma(2)$, say. For fixed $y$ with $|y|\le 1$, let $z=x-\lambda\rho_k(x)y$. Then
$$
\frac{\p z^i}{\p x^j}=\delta_{ij}-y^i\lambda\frac{\p\rho_k}{\p x^i}.
$$
By \eqref{e-cutoff-1}, we can  Choose $\lambda>0$ small enough independent of $\e$ and $k$ so that
$$2\ge \det(\delta_{ij}-\lambda y^i\frac{\p\rho_k}{\p x^i} )\ge \frac12,
 $$
 and so that $z=z(x)$ is a diffeomorphism with the Jacobian being bounded above and below by some constants independent of $\e, k$. Hence
\bee
\begin{split}
\lf(\int_{\Sigma(1)\cap U_k}|v|^p(x)dx\ri)^\frac1p\le &\lf[\int_{\Sigma(1)\cap U_k}\lf(\int_{\R^n}|u(x-\lambda\rho_k(x)y)|\varphi(y)dy\ri)^pdx\ri]^\frac1p\\
\le &\int_{B(1)}\varphi(y)\lf(\int_{\Sigma(1)\cap U_k}|u(x-\lambda\rho_k(x)y)|^pdx\ri)^\frac1p dy\\
\le &C_1\lf(\int_{\Sigma(2)}|u(z)|^p dz\ri)^\frac1p
\end{split}
\eee
for some constant $C_1$ independent of $\e, k$ provided $\e$ is small enough.
Now, for $x\notin \Sigma(2\e)$, then $v(x)=u(x)$ and if $x\in \Sigma(\e)$, then $v(x)$ is the standard mollification. If $x\in  \Sigma(2\e)\setminus \Sigma(\e)$, then
\bee
|\p v|(x)\le \int_{\R^n}|\p u|(x-\lambda\rho_k(x)y)\lambda|\p \rho_k(x)|\varphi(y)dy.
\eee
Since $|\p\rho_k|$ is bounded by \eqref{e-cutoff-1}, we can prove as before that
$$
\lf(\int_{\Sigma(1)\cap U_k}|\p v|^p(x)dx\ri)^\frac1p\le C_2\lf(\int_{\Sigma(2)}|\p u|^p (z) dz\ri)^\frac1p
$$
for some constant $C_2$ independent of $\e, k$ provided $\e$ is small enough. This completes the proof of the lemma.

\end{proof}

In  addition to ({\bf a1}) and ({\bf a2}), assume
\begin{enumerate}
  \item [({\bf a3})] The scalar curvature $\mathcal{S}_{g_0}$ of $g_0$ satisfies $\mathcal{S}_{g_0}\ge\sigma$ in $M\setminus\Sigma$, where $\sigma$ is a constant.
\end{enumerate}

  We want to modify $g_{\e,0}$ to obtain a smooth metric with scalar curvature bounded below by $\sigma$. We  first consider the case that $M$ is compact. Let $\e_0>0$ be small enough so that for all $\e_0\ge\e>0$,
$$
(1+\e(n))^{-1}g_{\e_0,0}\le g_{\e,0}\le (1+\e(n)) g_{\e_0,0},
$$
where $\e(n)>0$ is the constant   depending only on $n$ in Theorem \ref{t-Simon}.  Hence if we let $h=g_{\e_0,0}$, then the $h$-flow has solution $g_\e(t)$ on $M\times[0,T]$ for some $T>0$ independent of $\e$,  with initial data $g_{\e,0}$ in the sense that $\lim_{t\to0}g_{\e}(x,t)=g_{\e,0}(x)$ uniformly in $M$, see Theorem \ref{t-Simon}. The curvature and all the covariant derivatives of curvature of $h$ are bounded because $M$ is compact.

 By \cite{Simon2002} and Lemmas \ref{l-gradient-1}, \ref{l-gradient-2}, \ref{l-approx-1} we have the following:
\begin{lma}\label{l-approx-2} Let $M$ be compact and $g_0$ satisfies ({\bf a1})--({\bf a3}). Suppose $p>n$. Let $\delta=\frac np<1$. Then
$$|^h\nabla g_\e(t)|_h^2\le Ct^{-\delta}, \ \ |^h\nabla^2g_\e(t)|^2\le Ct^{-1-\delta}
$$
for some constant $C$ independent of $\e$, $t$.  Moreover, $g_\e(t)$ subconverge to the solution $g(t)$ of the $h$-flow with initial data $g_0$ in $C^\infty$ norm in compact sets of $M\times(0,T]$ and in compact sets of $M\setminus \Sigma\times[0,T]$.
\end{lma}

For $\e>0$ small enough, let
\be\label{e-DeTurck-1}
W^k=(g_\e(t))^{pq}\lf(\Gamma_{pq}^k(g_\e(t))-\Gamma_{pq}^k(h)\ri),
\ee
and let $\Phi_t$ be the diffeomorphism given by
\be\label{e-DeTurck-2}
\frac{\p}{\p t}\Phi_t(x)=-W(\Phi_t(x),t); \ \ \Phi_0(x)=x.
\ee
Let $\wt g_\e(t)=\Phi_t^*g_\e(t)$. Then $\wt g_\e(t)$ satisfies the Ricci flow equation with initial data $g_{\e,0}$. Note that $W$ and $\Phi_t$ depend also on $\e$. Recall the Ricci flow equation is:
\be\label{e-Ricciflow}
\frac{\p}{\p t}g_{ij}=-2R_{ij}.
\ee

\begin{lma}\label{l-hflow-2} Same assumptions and notation as  in Lemma \ref{l-approx-2}. For $\e$ small enough,   $|W|_h\le Ct^{-\frac12\delta}$, $|\Rm(\wt g_\e(t))|\le Ct^{-\frac12(1+\delta)}$  and

$$
C^{-1}h\le g_\e(t)\le Ch
$$
for some $C$, independent of $\e, t$.
\end{lma}
\begin{proof} The bound of $W$ is given by Lemma \ref{l-approx-2}. Since the bound of curvature is unchanged under diffeomorphism, $|\Rm(\wt g_\e(t))|\le Ct^{-\frac12(1+\delta)}$ by Lemma \ref{l-approx-2}. From this we conclude from the Ricci flow equation that $\wt g_\e(t)$ is uniformly  equivalent to $g_{0,\e}$ which is uniformly equivalent to $h$.
\end{proof}
\begin{lma}\label{l-codimone-1}
Let $\mathcal{S}(t)$ be the scalar curvature of $g(t)$. Then there is $C>0$ independent of $t, \e$ such that
$$
\exp(-Ct^{\frac12(1-\delta)})\int_M (\mathcal{S}(t)-\sigma)_-dv_{g(t)}
$$
is nonincreasing in $(0,T]$,  where $f_-=\max\{-f,0\}$  is the negative part of $f$.
\end{lma}
\begin{proof} As in \cite{McFeronSzekelyhidi2012},  fix $\theta>0$,
 for $\e>0$, let
  $$
  v=\lf((\mathcal{S}_\e(t)-\sigma)^2+\theta\ri)^\frac12-\lf(\mathcal{S}_\e(t)-\sigma\ri)
  $$
  where $\mathcal{S}_\e(t)$ is the scalar curvature of $\wt g_\e(t)$. Let $\Delta$ and $\nabla$ be the Laplacian and covariant derivative with respect to $\wt g_\e(t)$. Use the evolution equation of the scalar curvature in Ricci flow, we have

  \bee
\begin{split}
\lf(\frac{\p}{\p t}-\Delta\ri)v=&\lf(\frac{\mathcal{S}_\e(t)-\sigma}
{\lf((\mathcal{S}_\e(t)-\sigma)^2+\theta\ri)^\frac12}-1\ri)\lf(\frac{\p}{\p t}-\Delta\ri)\mathcal{S}_\e(t)-\frac{\theta|\nabla \mathcal{S}_\e|^2}{\lf((\mathcal{S}_\e(t)-\sigma)^2+\theta\ri)^\frac12}\\
=&\lf(\frac{\mathcal{S}_\e(t)-\sigma}
{\lf((\mathcal{S}_\e(t)-\sigma)^2+\theta\ri)^\frac12}-1\ri)\cdot 2|\nabla \Ric(t)|^2-\frac{\theta|\nabla \mathcal{S}_\e(t)|^2}{\lf((\mathcal{S}_\e(t)-\sigma)^2+\theta\ri)^\frac32}\\
\le &0
\end{split}
\eee
where $\Ric(t)$ is the Ricci tensor of $\wt g_\e(t)$. Using Lemma \ref{l-hflow-2}  we have

 \be\label{e-neg-part-2}
\begin{split}
\frac{d}{dt}\int_M   vdv_{\wt g_\e(t)}=&\int_M   \frac{\p}{\p t}vdv_{\wt g_\e(t)}-\int_M\mathcal{S}_\e(t)  vdv_{\wt g_\e(t)}\\
\le &\int_M   \Delta  vdv_{\wt g_\e(t)}+C_1t^{-\frac12(1+\delta)}\int_M   vdv_{\wt g_\e(t)}\\
=& C_1t^{-\frac12(1+\delta)}\int_M   vdv_{\wt g_\e(t)}\\
\end{split}
\ee
for some constant $C_1$ independent of $t, \e$. From this and let $\theta\to0$, we conclude that for some constant $C$ independent of $t$ and $\e$:
\bee
\exp(-Ct^{\frac12(1-\delta)})\int_M (\mathcal{S}_\e(t)-\sigma)_-dv_{\wt  g_\e(t)}
\eee
is nonincreasing in $(0,T]$.  Noting that $\wt g_\e(t)=\Phi_t^*(g_\e(t))$, by Lemma \ref {l-approx-2} let $\e\to0$, the result follows.
\end{proof}

We first consider the case that the codimension of $\Sigma$ is at least 2 in the following sense.

\begin{enumerate}
  \item [({\bf a4})] The volume $V(\Sigma(\e),g_0)$ with respect to $g_0$ of the $\e$-neighborhood $\Sigma(\e)$ of $\Sigma$ is bounded by $C\e^2$ for some constant $C$ independent of $\e$. Here
      $$
      \Sigma(\e)=\{x\in M|\ d_{g_0}(x,\Sigma)<\e\}.
      $$
\end{enumerate}

\begin{lma}\label{l-codimtwo-2} With the same assumptions and notation as in Lemma \ref{l-approx-2}. Suppose ({\bf a4}) is true.    Then
$S(t)\ge\sigma$ for all $t>0$.
\end{lma}
\begin{proof} By Lemma \ref{l-codimone-1}, it is sufficient to prove that:
\be\label{e-approx-3}
\lim_{t\to 0}\int_M (\mathcal{S}(t)-\sigma)_-dv_{g(t)}=0.
\ee

For any $\e>0$, let $\Phi_t$ be the diffeomorphisms as before so that $\wt g_\e(t)=\Phi_t^*(g_\e(t))$ is the solution to the Ricci flow. For any $\theta>0$, let $v$ as in the proof of Lemma \ref{l-codimone-1}. Let
    $$\beta=\displaystyle{\frac1\e(\e-\sum_{k=1}^{N}\psi_k\rho_k)}.$$
    We may modify $\rho_k$ so that if $\e$ is small enough then $\beta$ is a smooth function on $M$ so that $\beta=0$ in $\Sigma(2\e)$, $ \beta=1$ outside $\Sigma(4\e)$, $0\le \beta\le 1 $, $|^h\nabla \beta|\le C\e^{-1}$,   and $|^h\nabla^2\beta|\le C \e^{-2} $ for some constant $C$ independent of   $\e, t$. Let
$$
\wt\beta(t,x)=\beta(\Phi_t(x)).
$$
Then
\bee
\begin{split}
\frac{d}{dt}\int_M \wt\beta^2 vdv_{\wt g_\e(t)}=&  \int_M v  \frac{\p}{\p t}(\wt\beta^2)   dv_{\wt g_\e (t)}+\int_M \wt\beta^2 \frac{\p}{\p t}vdv_{\wt g_\e(t)}-\int_M\mathcal{S}_\e(t)\wt\beta^2 vdv_{\wt g_\e(t)}\\
\le &  \int_M v \frac{\p}{\p t}(\wt\beta)^2   dv_{\wt g_\e (t)}+\int_M \wt\beta^2 \Delta_{\wt g_\e(t)} vdv_{\wt g_\e(t)}\\
&+C_1t^{-\frac12(1+\delta)}\int_M \wt\beta^2 vdv_{\wt g_\e(t)}\\
=&I+II+C_1t^{-\frac12(1+\delta)}\int_M \wt\beta^2 vdv_{\wt g_\e(t)}.
\end{split}
\eee
for some constant $C_1>0$ independent of $t,\e, \theta$ by Lemma \ref{l-hflow-2}. Let $w(y)=v(\Phi_t^{-1}(y))$.
  Since in local coordinates,
$$
\Delta_{g_\e(t)}f=g_\e^{ij}\lf(\p_i\p_j f-\Gamma_{ij}^k \p_k\ri)
$$
with $|\Gamma_{ij}^k|\le Ct^{-\frac\delta2}$ for some constant independent of $\e, t$ by Lemma \ref{l-approx-2}, we have
\bee
\begin{split}
II=&\int_M \beta^2 \Delta_{ g_\e(t)} w dv_{ g_\e(t)}\\
=&\int_M w\Delta_{g_\e(t)}(\beta^2) dv_{g_\e(t)}\\
\le& C_2\int_{\Sigma(2\e)} w|\e^{-2}+\e^{-1}t^{-\frac\delta2}\beta| dv_{g_\e(t)} \\
\le& C_3\lf(t^{-\frac12(1+\delta)}+\e^{-1}t^{-\frac\delta2-\frac14(1+\delta)}\int_{\Sigma(4\e)}\beta w^\frac12dv_{g_\e(t)}\ri)\\
\le&C_4\lf[t^{-\frac12(1+\delta)}+ t^{-\frac14(1+3\delta)} \lf( \int_{M}\wt\beta^2 v dv_{\wt g_\e(t)}\ri)^\frac12\ri]
\end{split}
\eee
for some constants $C_2-C_4$ independent of $\e, t, \theta$,  where    we have used Lemma \ref{l-approx-2}, the fact that $\beta=1$ outside $\Sigma(4\e)$, H\"older inequality   and the fact that $V(\Sigma(4\e))=O(\e^2)$. To estimate $I$, we have
 \bee
 \begin{split}
 \frac{\p}{\p t} \wt\beta  =&(d\wt\beta)(\frac{\p}{\p t})\\
 =&d\beta\circ d\Phi_t(\frac{\p}{\p t})\\
 =&d\beta (W).
 \end{split}
\eee
Hence by Lemma \ref{l-approx-2},   we have
\bee
|\frac{\p}{\p t}\wt\beta|(x)\le C_5|^h\nabla \beta|(\Phi_t(x))|\le C_6\e^{-1}t^{-\frac\delta2}
\eee
for some constants $C_5, C_6$ independent of $\e, t, \theta$.
Hence if $w$ is as above, then
\bee
\begin{split}
I\le &C_6\e^{-1}t^{-\frac\delta2}\int_{\Sigma(4\e)}  \beta w(y)dv_{g_\e(t)}\\
\le &C_7 t^{-\frac14(1+3\delta)}\lf(\int_{\Sigma(4\e)}\wt\beta^2 v dv_{\wt g_\e(t)}\ri)^\frac12.
\end{split}
\eee
for some constant $C_7$ independent of $\e, t, \theta$.
    To summarize,  if we let
    $$
    F=\int_M \wt\beta^2 vdv_{\wt g_\e(t)},
    $$
    then
    \bee
    \begin{split}
    \frac{d F}{dt}\le &C_8\lf (t^{-\frac12(1+\delta)}+ t^{-\frac12(1+\delta)}F+t^{-\frac14(1+3\delta)}F^\frac12\ri)\\
    \le&C_8\lf(t^{-\frac12(1+\delta)}+t^{-\delta}+2t^{-\frac12(1+\delta)}F\ri)
    \end{split}
    \eee
 for some constant $C_8$ independent of $\e, t, \theta$. Integrating from $0$ to $t$, and let $\theta\to0$. Since $g_{\e,0}=g_0$ outside $\Sigma(2\e)$, $\Phi_0=$id, and $\beta=0$ on $\Sigma(2\e)$, and  $\mathcal{S}_{g_0} \ge\sigma$ outside $\Sigma$, there exist constants $C_9-C_{10}$ independent of $\e, t$
 $$
 \exp(-C_9t^{\frac12(1-\delta)})\int_M \wt\beta^2(\mathcal{S}_\e(t)-\sigma)_-dv_{\wt g_\e(t)}\le C_{10}\lf(t^{\frac12(1-\delta)}+t^{1-\delta}\ri)
 $$
 because $0<\delta<1$. Let $\e\to0$, we see that \eqref{e-approx-3} is true and the proof of the lemma is completed.

\end{proof}
By Lemmas \ref{l-approx-2} and \ref{l-codimtwo-2}, using $g(t)$  we have:
\begin{cor}\label{c-approx-1}
Let $(M^n,\mathfrak{b})$ be a smooth compact  manifold and let $g_0$ be a continuous Riemannian metric satisfying the following:
 \begin{enumerate}
   \item [(a)] There is a compact set $\Sigma$ such that $g_0$ is smooth on $M\setminus \Sigma$  with scalar curvature bounded below by $\sigma$.
   \item [(b)] $g_0$ is in $W_{\text{loc}}^{1,p}$ for some $p>n.$
   \item [(c)] $V(\Sigma(\e),g_0)=O(\e^2)$ as $\e\to0$, where $\Sigma(\e)=\{x\in M| d_{\mathfrak{b}}(x,\Sigma)<\e\}$.
 \end{enumerate}
  Then there exist a sequence of smooth metrics $g_i$ satisfying the following:   (i)   as $i$ tends to infinity $g_i$ converges to $g_0$ uniformly in $M$, and converges to $g_0$ in $C^\infty$ norm on any compact subset of $M\setminus \Sigma $; (ii) the scalar curvature $\mathcal{S}_i$ of $g_i$ satisfies $\mathcal{S}_i\ge\sigma$.
\end{cor}

\begin{remark}\label{r-codim} If the codimension of $\Sigma$ is only assumed to be larger than 1, then the conclusions of Lemma \ref{l-codimtwo-2} and Corollary \ref{c-approx-1} are still true under some additional assumptions on the second derivatives of $g_0$.

\end{remark}

Next let us consider the case that $\Sigma$ is an embedded hypersurface. Let $(M^n,g_0)$ be a Riemannian metric satisfying  the following:

\begin{enumerate}
\item[({\bf b1})] $\Sigma$ is a compact embedded orientable hypersurface, and $g_0$ is smooth on $M\setminus \Sigma$ with scalar curvature $\mathcal{S}_{g_0}\ge \sigma$.
  \item [({\bf b2})] There is neighborhood $U$ of $\Sigma$ and a smooth function $t$  defined near $U$  so that  $U$ is diffeomorphic to $\{  -a<t<a \}\times \Sigma$ for some $a>0$ with   $\Sigma=\{t=0\}$. Moreover, $g_0=dt^2+g_{\pm}(z,t)$, $z\in \Sigma$ so that $(t,z)$ are smooth coordinates and $g_-(\cdot,0)=g_+(\cdot,0)$, where $g_+$ is defined and smooth on $t\ge0$, $g_-$ is defined and smooth on $t\le 0$.

  \item [({\bf b3})] Let $U_+=\{t>0\}$, $U_-=\{t<0\}$.  With respect to the unit normal $\frac{\p}{\p t}$ the mean curvature $H_+$  of $\Sigma$ with respect to $g_+$ and the mean curvature $H_-$  of $\Sigma$ with respect to $g_-$ satisfies $H_-\ge H_+$.
\end{enumerate}

By \cite[Prop. 3.1]{Miao2002}, let $\e>0$ be small enough, one can find smooth metric $g_{\e,0}$ such that (i) $g_{\e,0}=g_0$ outside $U(\e)=\{-\e <t<\e\}$; (ii) $g_{0,\e}$ converges uniformly to $g_0$; (iii) $|^h\nabla g_{0,\e}|_h\le C$ for some fixed background smooth metric $g$; (iv) there exists a  $c>0$ independent of $\e$ such that the scalar curvature $\mathcal{S}_{g_{0,\e}}$ satisfies:

\be\label{e-scalar-1}
\left\{
  \begin{array}{ll}
    \mathcal{S}_{g_{0,\e}}=\mathcal{S}_{g_0}, & \hbox{outside $U(\e)$;} \\
    |\mathcal{S}_{g_{0,\e}}|\le c, & \hbox{in $\frac{\e^2}{100} \delta_i^2<|t|\le \e$};\\
 \mathcal{S}_{g_{0,\e}}(z,t)\ge  -c +(H_-(z)-H_+(z))\e^{-2}\phi(\frac{100t}{\e^2}), & \hbox{in $-\frac{\e^2}{100}  <|t|\le \frac{\e^2}{100}  $};\\
  |\mathcal{S}_{g_{0,\e}}|\le c\e^{-2};
  \end{array}
\right.
\ee
for $z\in \Sigma$. Here $\phi\ge0$ is a smooth function in $\R$ with compact support in $[-1/2,1/2]$ so that
$$
\int_\R \phi(s)ds=1.
$$

Similar arguments as before using $h$-flow, we can conclude:
\begin{cor}\label{c-approx-2}
Let $M^n$ be a  compact smooth manifold and let $g_0$ be a Riemannian metric satisfying {\bf   (b1)--(b3)} such that the scalar curvature of $g_0$ on $M\setminus\Sigma$ is at least $\sigma$. Then there exist a sequence of smooth metrics $g_i$    such that as $i$ tends to infinity $g_i$ converges to $g_0$ uniformly in $M$, and converges to $g_0$ in $C^\infty$ norm on any compact subset of $M\setminus \Sigma.$ Moreover, $\mathcal{S}_{g_i}\ge \sigma$.
\end{cor}
\begin{proof} As before, choose $h=g_{0,\e_0}$ for $\e_0$ small enough, one can solve the $h$-flow with initial data $g_{0,\e}$. Let $g_{\e}(t)$ be the solution and let $\mathcal{S}_\e(t)$ be its scalar curvature. From the proof of Lemma \ref{l-codimone-1}, one can conclude that
\bee
\begin{split}
\exp(-C_3t^\frac12)\int_M(\mathcal{S}_\e(t)-\sigma)_-dv_{g_\e(t)}\le& \int_M (\mathcal{S}_{g_{0,\e}}-\sigma)_-dv_{g_{0,\e}}\\
=&\int_{U(\e)}(\mathcal{S}_{g_{0,\e}}-\sigma)_-dv_{g_{0,\e}}\\
 \le& C_1\e
\end{split}
\eee
for some $C_3>0$ independent of $\e, t$. Here we have used the fact that $H_--H_+\ge0$. Let $\e\to0$, we conclude that the solution $g(t)$ of the $h$-flow with initial value $g_0$ has scalar curvature no less than $\sigma$. The result follows as before.
\end{proof}

\begin{remark}\label{r-Miao}By \cite{Miao2002}, suppose $\Sigma$ is a compact orientable hypersurface  a neighborhood of  $\Sigma$ is of the disjoint union of $U_1$, $U_2$ and $\Sigma$. Assume   $g_0$ is smooth up $\Sigma$ from each side $U_i$ of $\Sigma$ and such that  the mean curvatures $H_1, H_2$ with respect to unit normals in the two sides of $\Sigma$ satisfying $H_1+H_2\ge0$, where  unit normals are chosen to be outward pointing in each side. Then one can find a smooth structure so that {\bf (b2), (b3)} are true.
\end{remark}

We give some applications:

\begin{cor}\label{c-cone-1} Let $(M^n,g)$ be a compact manifold such that $M^n$ is topological $n$-torus, $g$ is smooth   except at a point where $g$ has a cone singularity of the form
$$
g=dr^2+\a^2 r^2h_0
$$
with $0<\a\le 1$ and $h_0$ is the standard metric on $\mathbb{S}^{n-1}$. Suppose the scalar curvature is of $g$ is nonnegative, then $g$ must be flat and $\a=1$.

\end{cor}
\begin{proof} For $r$ small, the mean curvature of the level set $\{r\}\times \mathbb{S}^{n-1}$ with respect to the normal $\p_r$ is $H=\frac{n-1}r$. Consider the Euclidean ball $B(\a r)$ of radius $\a r$ with center at the origin. Then metric of the boundary is $(\a r)^2h_0$. Moreover, the mean curvature is $H_0=\frac{n-1}{\a r}$. Since $\a\le 1$, $H_0\ge H$. By gluing $B(\a r)$ along with $M$ along $\{r\}\times \mathbb{S}^{n-1}$, we obtain a metric with corner so that {\bf(b1)--(b3)} are true by changing the smooth structure if necessary. Still denote this metric by $g$.  By Corollary \ref{c-approx-2}, there exist smooth metrics $g_i$ on the new manifold with nonnegative scalar curvature so that $g_i\to g$ in $C^\infty$ away from the singular part. By \cite{SchoenYau1979,SchoenYau1979-1,GromovLawson}, $g_i$ is flat.  Hence $g$ must be flat away from the singular part. Let $r\to0$, we conclude that the original metric $g$ is flat,  and we must have $\a=1$.
\end{proof}

Similarly, one can prove the following:

\begin{cor}\label{c-cone-2} Let $(M^n,g)$ be a compact manifold such that $M^n$ is topological $n$-torus, $g$ is smooth  away some compact set with codimension at least 2. Moreover, assume $g$ is in $W^{1,p}_{\text{loc}}$ for some $p>n$.  Suppose the scalar curvature is of $g$ is nonnegative, then $g$ must be flat.
\end{cor}

\begin{remark}\label{r-cone-1} Suppose $M$ is asymptotically flat with nonnegative scalar curvature and with some cone singularities as in Corollary \ref{c-cone-1}, then we still have positive mass for each end by the result in \cite{Miao2002}. The proof is similar. Compare this result with the example in Proposition \ref{p-cone-pm-2}.

\end{remark}

Let us consider the case that $M^n$ is noncompact. Let $g_0$ be a continuous Riemannian metric on $M$ which is smooth outside a compact set $\Sigma$. Suppose there is a family of smooth complete  metrics $g_{\e,0}$ on $M$ such that $g_{\e,0}$ converges uniformly to $g_0$ and converges smoothly on compact sets of $M\setminus \Sigma$. Assume $g_{\e,0}$ has bounded curvature for all $\e$. As before, we can find $\e_0>0$ such that if $h=g_{\e_0,0}$ then there are solutions $g_\e(t)$ to the $h$-flow with initial data $g_{\e,0}$, and solution to the $h$-flow with initial data $g(t)$ on some fixed interval $[0,T]$, $T>0$. As in \cite{Simon2002}, using \cite{Shi1989}, we may assume that all the derivatives of the curvature of $h$ are bounded. Moreover,   $g_\e(t)$ converge uniformly on compact sets of $M\times(0,T]$ and $M\setminus\Sigma\times[0,T]$. Suppose the scalar curvature of $g_0$ satisfies $\mathcal{S}_{g_0}\ge \sigma$. We want to find conditions so that the scalar curvature of $g(t)$ is also bounded below by $\sigma$.

\begin{lma}\label{l-noncompact-1} With the above assumptions and notation, suppose
\begin{enumerate}
  \item [(i)] $g_{\e,0}=g_0$ outside $\Sigma(2\e)$.
  \item [(ii)] $|^h\nabla g_{\e}(t)|\le Ct^{-\frac\delta 2}$, $|^h\nabla^2 g_{\e}(t)|\le Ct^{-\frac12(1+\delta)}$ for some $C$ independent of $\e, t$.
  \item [(iii)] There is $R_0>0$ and $C>0$ independent of $\e, t$ such that
  $$
  \int_{M\setminus B(o,R_0)}|\mathcal{S}_\e(t)-\sigma|dv_h\le C.
  $$
 where $B(o,R_0)$ is the geodesic ball with respect to $h$ and $\mathcal{S}_\e(t)$ is the scalar curvature of $g_\e(t)$.
 \item[(iv)] $V(\Sigma(2\e), g_0)=O(\e^2)$.

\end{enumerate}
Then the scalar curvature $\mathcal{S}(t)$ of $g(t)$ satisfies $\mathcal{S}(t)\ge \sigma$ for all $t>0$.
\end{lma}
\begin{proof} By \cite{Shi1989,Tam2010}, we can find a smooth function $\rho$ such that
$$
C_1^{-1}(r(x)+1)\le \rho(x)\le C_1(1+r(x))
$$
for some constant $C_1>0$ where $r(x)$ is the distance function to a fixed point $o$ with respect to $h$. Moreover, the gradient and Hessian of $\rho$ with respect to $h$ are uniformly bounded.

Let $0\le \eta\le1$ be a smooth function on $\R$ so that $\eta=1$ on $[0,1]$ and $\eta=0$ on $[2,\infty)$. We proceed as in the proofs of Lemmas \ref{l-codimone-1}, \ref{l-codimtwo-2}. For $R>>1$, denote $ \eta(\rho(x)/R)$ still by $\eta(x)$. Let $\wt g_\e$ be the Ricci flow corresponding to the $g_\e(t)$ and let $\mathcal{S}_\e(t)$ be its scalar curvature. Let $\theta>0$ and let $v$ as in the proof of Lemma \ref{l-codimone-1}, we have
\bee
\begin{split}
&\frac{d}{dt}\int_M \eta vdv_{\wt g_\e(t)}\le\\
 &C_2\lf(t^{-\frac12(1+\delta)}\int_M \eta v dv_{\wt g_\e(t)}+\int_M v|\Delta \eta| dv_{\wt g_\e(t)}\ri)\\
\le &C_3\lf(t^{-\frac12(1+\delta)}\int_M \eta v dv_{\wt g_\e(t)}+t^{-\frac \delta2}R^{-1}\int_{M\setminus B(o, 2C_1R)} (|\mathcal{S}_\e(t)-\sigma|+\theta) dv_{\wt g_\e(t)}\ri)
\end{split}
\eee
for some positive constants $C_2, C_3$ independent of $t,\e, \theta$. Hence

$$
\frac{d}{dt}\lf(\exp(-C_4t^{\frac12(1+\delta)})\int_M \eta vdv_{\wt g_\e(t)}\ri)\le C_5t^{-\frac \delta2}R^{-1}\int_{M\setminus B(o, 2C_1R)} (|\mathcal{S}_\e(t)-\sigma|+\theta) dv_{\wt g_\e(t)}
$$
for  some positive constants $C_4, C_5$ independent of $t,\e, \theta$. Integrating from $0<t_1<t_2$,  let $\theta\to0$ and then let $R\to\infty$, using condition (iii), we conclude that
$$ \exp(-C_4t^{\frac12(1+\delta)})\int_M \lf(\mathcal{S}_\e(t)-\sigma\ri)_-dv_{\wt g_\e(t)}
$$
is nonincreasing in $t$. Let $\e\to0$, we conclude that
 $$ \exp(-C_4t^{\frac12(1+\delta)})\int_M \lf(\mathcal{S}(t)-\sigma\ri)_-dv_{g_\e(t)}
$$
is nonincreasing in $t$.

Next we proceed as in the proof of Lemma \ref{l-codimtwo-2}. But we need the cutoff function $\eta$. For $\e>0$ and $\theta>0$ as in the proof of Lemma \ref{l-codimtwo-2}, let $\beta, \wt\beta$ as in that proof, we have for $R>>1$,
\be\label{e-approx-4}
\begin{split}
\frac{d}{dt}F dv \le& C_6 \lf(t^{-\frac12(1+\delta)}+t^{-\delta}+t^{-\frac12(1+\delta)}F+\int_M|\Delta\eta|v\wt\beta^2\ri)\\
\le &C_7\lf(t^{-\frac12(1+\delta)}+t^{-\delta}+t^{-\frac12(1+\delta)}F+\frac1R\int_{M\setminus B(o,2C_1R)} (|\mathcal{S}_\e(t)-\sigma|+\theta)dv_{\wt g_\e(t)}\ri)
\end{split}
\ee
for some constants $C_6, C_7$ independent of $\e, t,\theta$ where
$$
F=\int_M\eta\wt\beta^2vdv_{\wt g_\e(t)}.
$$
Integrating from 0 to $t$ and let $\theta\to0$, we have
$$
\int_M\eta\wt\beta^2(\mathcal{S}_\e(t)-\sigma)_- dv_{\wt g_\e(t)} \le C_8\lf(t^{1-\delta}+t^{\frac12(1-\delta)})+\frac1R\int_0^t\int_{M\setminus B(o,2C_1R)} (|\mathcal{S}_\e(s)-\sigma|)dv_{\wt g_\e(s)} )ds\ri)
$$
for some constant $C_8$   independent of $\e, t$. Here we have used the fact that $g_{\e,0}=g_0$ outside $\Sigma(2\e)$ and the fact that $\mathcal{S}_{g_0}\ge \sigma$.
Let $R\to\infty$, using (iii), and finally let $\e\to0$, we conclude that
$$
\int_M(\mathcal{S}(t)-\sigma)_-dv_{g(t)}\le C_8(t^{1-\delta}+t^{\frac12(1-\delta)}) .
$$
Since
$$
\exp(-C_4t^{\frac12(1+\delta)})\int_M \lf(\mathcal{S}(t)-\sigma\ri)_-dv_{g_\e(t)}
$$
is nonincreasing in $t$, we conclude that the lemma is true.
\end{proof}


\section{singular metrics realizing the nonpositive Yamabe invariant}\label{s-Einstein-1}

In this section, we will apply the results in previous sections to study singular metrics on compact manifold. Let $M^n$ be a compact smooth manifold without boundary. Then as in the Introduction, we may define  the {\it Yamabe invariant} $\sigma(M)$. It is well-known that if $\sigma(M)\le0$ and if $g$ is a smooth metric which realizes $\sigma(M)$, then $g$ is Einstein, see \cite[p.126--127]{Schoen1989} for example. If $\sigma(M)>0$, the situation is more complicated,  for some recent results  see \cite{Macbeth2014}.

 In this section we want to discuss the following question:

{\it Suppose $g$ is a continuous Riemannian metric on $M$ which is smooth outside some compact set $\Sigma$ so that the volume of $g$ is normalized to be 1. Suppose the scalar curvature of $g$ satisfies $\mathcal{S}_g\ge  \sigma(M)$ away from $\Sigma$. What can we say about $g$?}

In the case that $\Sigma$ has codimension at least 2, we have the following:
\begin{thm}\label{t-Yamabe-1} Let $M^n$ be a smooth compact manifold such that $\sigma(M)\le 0$. Suppose $g_0$ is a Riemannian metric with $V(M,g_0)=1$  satisfying the following:
\begin{enumerate}
  \item [(i)] There is a compact subset $\Sigma$ such that $g_0$ is smooth on $M\setminus \Sigma$ with scalar curvature $\mathcal{S}_{g_0}\ge\sigma(M)$ away from $\Sigma$.
  \item [(ii)] $g_0$ is in $W_{\text{loc}}^{1,q}$ for some $q>n$ in the sense that $g_0$ has weak derivative and $|g_0|_\mathfrak{b}$, $|^\mathfrak{b}\nabla g_0|_\mathfrak{b}\in L^q_{\text{loc}}$ with respect to a smooth background  metric $\mathfrak{b}$.
\item [(iii)] The volume $V(\Sigma(\e),g_0)$ with respect to $g_0$ of the $\e$-neighborhood $\Sigma(\e)$ of $\Sigma$ is bounded by $C\e^2$ for some constant $C$ independent of $\e$. Here
      $$
      \Sigma(\e)=\{x\in M|\ d_{g_0}(x,\Sigma)<\e\}.
      $$
\end{enumerate}

    Then $g_0$ is Einstein on $M\setminus \Sigma$.
\end{thm}

Let $(M^n,g_0)$ be as in the theorem. Let   $$\ring{\Ric}(g_0)=\displaystyle{\Ric(g_0)-\frac{\mathcal{S}_0}n g_0}$$
 be the traceless part of $\Ric(g_0)$ where $\mathcal{S}_0=\mathcal{S}_{g_0}$ is the scalar curvature of $g_0$. Let $x_0\in M\setminus \Sigma$. We want to prove that $\ring{\Ric}(x_0)=0$.  Suppose $\ring{\Ric}(g_0)(x_0)\neq0$, then there is $r>0$ such that $B_{x_0}(4r;g_0)\cap \Sigma=\emptyset$ and there is $c>0$, $|\ring{\Ric}(g_0)|(x_0)\ge2c$ in $B_{x_0}(3r)$. By Corollary \ref{c-approx-1}, we can find smooth metrics $g_i$ such that (i) $g_i$ converges uniformly to $g_0$ and converges in $C^\infty$ norm on any compact sets in $M\setminus \Sigma$; (ii) $V(M,g_i)=1$; (iii)   the scalar curvature  $\mathcal{S}_i$ of $g_i$ satisfies $\mathcal{S}_i\ge\sigma-\delta_i$ for all $i$ with $\delta_i\downarrow0$.  Hence we may assume that

\be\label{e-Einstein-1}
|\ring{\Ric}(g_i)|(x)\ge c
\ee
in $B_{x_0}(2r;g_i)$ for all $i$, and  $B_{x_0}(r;g_i)\subset B_{x_0}(2r;g)$, $B_{x_0}(2r;g_i)\subset B_{x_0}(3r;g)$. We may also assume that the distance function $r_i(x)$ from $x_0$ with respect to $g_i$ are smooth in $B_{x_0}(3r;g)$,  provided $r>0$ is small enough, independent of $i$.

Let $\phi$ be smooth function on $[0,\infty)$ with $\phi\ge 0$, $\phi=1$ on $[0,1]$ and $\phi=0$ on $[2,\infty)$ and such that $|\phi'|^2\le C\phi$, with $C$ being an absolute constant. Let $$h_i(x)=\phi\lf(\frac{r_i(x)} r\ri)\ring{\Ric}(g_i)(x).$$

For   $|\tau|>0$, let $G_{i;\tau}=g_i+\tau h_i$. Then   there is $\tau_0>0$ such that $G_{i;\tau}$ are smooth metrics for all $i$ and for all $0<|\tau|\le\tau_0$.

  In the following, $E_k=E_k(x,\tau)$ ($k=1, 2$) will denote a quantity such that $|E_k|\le C|\tau|^k$ for some $C$ independent of $i$ and $\tau$.

\begin{lma}\label{l-vol-1} $dv_{G_{i;t}}=dv_{g_i}(1+E_2)$ and $V(M, G_{i;t})=1+E_2$, here $dv_g$  denots the volume element of metric $g$

\end{lma}
\begin{proof} Since $g_i\to g$  uniformly   on compact sets of $M\setminus\Sigma$ in $C^\infty$ norm and since $h_i$ is traceless, the results follow.

\end{proof}

We have the following general fact, see \cite[Prop. 4]{BrendleMarques2011}:
\begin{lma}\label{l-scalar-curvature-1}
Let $(\Omega^n,  g)$ be a smooth Riemannian manifold. Let $\bar g=  g+h$ with $ |h|_{ g}\le \frac12$, then the scalar curvatures are related as:
 \bee
 \mathcal{S}_{\bar g}-\mathcal{S}_g= \div_g(\div_g(h))- \Delta_g\tr_gh-\la h,\Ric(g)\ra_g+F
 \eee
 where $$|F|\le C\lf( |\nabla h|^2+|h|_g|\nabla^2h|_g+|\Ric(g)||h|^2_g\ri)$$ for some constant $C$ depending only on $n$. Here $\nabla$ is the covariant derivative with respect to $ g$.
\end{lma}
 \begin{lma}\label{l-scalar-curvature-2} Let $\mathcal{S}_i$ be the scalar curvature of $g_i$ and $\mathcal{S}_{i;\tau}$ be the scalar curvature of $G_{i;\tau}$. Then
\bee
\label{e-scalar-1}
\begin{split}
 {\mathcal{S}}_{i;\tau}=\mathcal{S}_i+\tau\div_{g_i}(\div_{g_i} h_i)-\tau\la h_i,\Ric(g_i)\ra_{g_i}+E_2(\tau).
\end{split}
\eee
  ${\mathcal{S}}_{i;\tau}=\mathcal{S}_i$ outside $B_{x_0}(2r,g_i)$ and  is bounded below by a constant independent of $i, \tau$.
\end{lma}
\begin{proof} The lemma follows from Lemma \ref{l-scalar-curvature-1}, the fact that $h_i$ is traceless,  $h_i=0$ outside $B_{x_0}(2r,g_i)$, the fact that $g_i\to g$ in $C^\infty$ outside $\Sigma$ and the fact that $\mathcal{S}_i\ge\sigma-\delta_i$.
\end{proof}

In the following, let
\be\label{e-constant}
a=\frac{4(n-1)}{n-2};\ \ p=\frac{2n}{n-2}.
\ee
By the resolution of the Yamabe conjecture \cite{Yamabe1960,Trudinger1968,Aubin1976-1,Schoen1984}, for each $i, \tau$, we can find smooth positive solution  $u_{i;\tau}$ of satisfies:

\be\label{e-Yamabe-1}
-a\Delta_{G_{i;\tau}}u_{i;\tau}+{\mathcal{S}}_{i;\tau}u_{i;\tau}= \lambda_{i;\tau}V_{i;\tau}^{-\frac2n}u_{i;\tau}^{p-1}.
\ee
with $ \lambda_{i;\tau}=Y(\mathcal{C}_{i,\tau})$ which is less than or equal to $\sigma$, in particular, it is nonpositive,  where $\mathcal{C}_{i,\tau}$ is the class of smooth metrics conformal to $G_{i;\tau}$. Moreover, $u_{i;\tau}$ is normalized by
$$
 \int_M u_{i;\tau}^pdv_{G_{i;\tau}}=1,
 $$
 and $V_{i,\tau}=V(M,G_{i;\tau})$.

\begin{lma}\label{l-Einstein-1} There is $0<\tau_1\le \tau_0$ independent of $i$ such that if $0>\tau\ge -\tau_1$, then
\bee
\begin{split}
\frac a2\int_M|^{(i;\tau)}\nabla u_{i;\tau}|^2_{G_{i;\tau}}dv_{G_{i;\tau}}&-\lambda_{i;\tau}V_{i;\tau}^{-\frac2n}+\sigma\\
\le& -C|\tau|\int_{B_{x_0}(2r,g_i)}\phi u_{i;\tau}^2   dv_{g_i}+C'\delta_i+E_2(\tau)
\end{split}
\eee
for some positive constants $C, C'$ independent of $i$ and $\tau$.
Here $^{(i;\tau)}\nabla$ is the covariant derivative with respect to $G_{i;\tau}$.
\end{lma}
\begin{proof} For simplicity of notations, in the following we denote $^{(i;\tau)}\nabla$ by $\nabla$, $G_{i;\tau}$ by $G$; $g_i$ by $g$; $u_{i;\tau}$ by $u$; $\lambda_{i;\tau}$ by $\lambda$; $\mathcal{S}_{i;\tau}$ by $\mathcal{S}_G$; $\mathcal{S}_{i}$ by $\mathcal{S}_g$;  and $V_{i;\tau}$ by $V$.

Multiply \eqref{e-Yamabe-1} by $u$ and integrating by parts, using the fact that $\int_M u^p dv_G=1$, we have

\be\label{e-E-1}
 \begin{split}
 a\int_M|\nabla u|^2_{G}dv_{G}- \lambda V^{-\frac2n}
=&- \int_M \mathcal{S}_{G}u^2 dv_{G}\\
\le&-\int_M {\mathcal{S}}_{G}u^2dv_g+E_2(\tau)\int_M u^2 dv_g\\
\end{split}
\ee
by Lemmas \ref{l-vol-1}, \ref{l-scalar-curvature-2}, and the fact that $g_i$ converges in $C^\infty$ norm in $B_{x_0}(3r,g_0)\supset B_{x_0}(g_i,2r)$. On the other hand, by Lemma \ref{l-scalar-curvature-2},   for any $0<\e<1$,
\be\label{e-E-2}
 \begin{split}
 -\int_M {\mathcal{S}}_{G}u^2dv_g
\le &-\int_M {\mathcal{S}}_{g}u^2dv_g-\tau\int_M\lf(\div_g(\div_gh)-\la h,\Ric(g)\ra_g \ri)u^2dv_g\\
&+E_2(\tau)\int_{B_{x_0}(2r;g)}u^2 dv_g\\
\le & -\int_M {\mathcal{S}}_{g}u^2dv_g+C_1|\tau|\int_Mu|^g\nabla u|_g\lf(|\phi'||\ring{\Ric}(g)|_g+\phi|^g\nabla \mathcal{S}_0|_g\ri)dv_g\\
&-|\tau|\int_M\phi|\ring{\Ric}(g)|^2  u^2dv_g
 +E_2(\tau)\int_{B_{x_0}(2r;g)}u^2 dv_g\\
\le&(-\sigma+\delta)\int_Mu^2dv_g+\lf(C_2+\e^{-1}\ri)|\tau|\int_M|^g\nabla u|_g^2dv_g\\&- C_3|\tau|\int_M\phi|\ring{\Ric}(g)|^2  u^2dv_g
 +\lf(E_2(\tau)+C_2\e|\tau|\ri)\int_{B_{x_0}(2r;g)}\phi u^2 dv_g\\
 \le &(-\sigma+\delta)\int_Mu^2dv_g+\lf(C_2+\e^{-1}\ri)|\tau|\int_M|^g\nabla u|_g^2dv_g\\&
 +\lf(E_1(\tau)+C_2\e-C_3 c\ri)|\tau|\int_{B_{x_0}(2r;g)}\phi u^2 dv_g
\end{split}
\ee
for some constants $C_1, C_2, C_3>0$ independent of $i,\tau$. Here we have used the fact that $|\phi'|^2\le C\phi$ and the fact that $\mathcal{S}_g\ge  \sigma-\delta_i$ which is negative, where we denote $\delta_i$ by $\delta$.       Choose $\e>0$ so that $C_2\e=\frac 12 C_3c$, then result follows if $\tau_1>0$ is small enough independent of $i$, by   \eqref{e-E-1}, \eqref{e-E-2}, H\"older inequality,   the fact that $g, G$ are uniformly equivalent, the fact that $\int_M u^pdv_G=1$, $V(M,g)=1$, and $V(M,G)=1+E_2(\tau)$.

\end{proof}

Let $0>\tau_k>-\tau_1$, $\tau_k\to0$. Since $\delta_i\to0$, for each $k$ we can find $i_k$ such that $ \delta_{i_k}\le \tau_k^2 $, $i_k\to\infty$. Let us denote $G_{i_k;\tau_k}$ by $G_k$, and $u_{i_k;\tau_k}$ by $u_k$. We want to prove the following:
\begin{lma}\label{l-uconvergence-1}
There is a constant $C>0$ such that for all $k$, $$\inf_{B_{x_0}(3r,g_0)}u_k\ge C.$$
\end{lma}

Suppose the lemma is true then we will have a contradiction. In fact, if we denote $\delta_{i_k}$ by $\delta_k$, since $V(M,G_k)=1+E_2(\tau_k)$, $\lambda\le\sigma$, by Lemma \ref{l-Einstein-1}, we have

\bee
\begin{split}
\frac a2\int_M|^{G_k}\nabla u_k|^2_{G_{k}}dv_{G_{k}}\le& -C_1|\tau_k|\int_{B_{x_0}(2r,g_{i_k})}\phi u_k^2   dv_{g_{i_k}}+C_2\delta_k+C_2\tau_k^2\\
\le& -C_1|\tau_k|\int_{B_{x_0}(2r,g_{i_k})}\phi u_k^2   dv_{g_{i_k}}+(C_2+1)  \tau_k^2
\end{split}
\eee
 for some positive constants $C_1, C_2$ independent of $k$. By Lemma \ref{l-uconvergence-1}, this is impossible if $k$ is large enough. Hence $\ring{\Ric}(g_0)(x_0)$ must be zero.  Theorem \ref{t-Yamabe-1} then follows.

It remains to prove Lemma \ref{l-uconvergence-1}.
Consider the equation:
\be\label{e-conformal-1}
-a\Delta u+\mathcal{S}u=\lambda u^{p-1}.
\ee

\begin{lma}\label{l-Lq-1} Let $(M^n,g)$ be a smooth metric with scalar curvature $\mathcal{S}\ge -s_0$, with $s_0\ge0$. Let $u>0$ be a solution of \eqref{e-conformal-1} with $||u||_p=1$ and with $\lambda\le 0$. Then for any $q>p$,
$$
||u||_q\le C(s_0, V(M ;g), n, q).
$$

\end{lma}
\begin{proof} This is from \cite{Trudinger1968}, see also \cite[Prop. 4.4]{LeeParker1987}. We sketch the proof here. Let $\theta>0$. Multiply \eqref{e-conformal-1} by $u^{1+2\theta}$ and integrating by parts, we have
\bee
\begin{split}
0\ge& \int_M \lambda u^{1+2\theta+p-1}dV_g\\
=& \int_M \lf(-a u^{1+2\theta}\Delta u+\mathcal{S}u^{2+2\theta}\ri)dv_g\\
=&\int_M\lf(a(1+2\theta)u^{2\theta}|\nabla u|^2+\mathcal{S}u^{2+2\theta}\ri)dv_g.
\end{split}
\eee
Let $w=u^{1+\theta}$, we have
\be\label{e-Lq}
\begin{split}
\int_M|\nabla w|^2 dv_g\le& -\frac{(1+\theta)^2}{a(1+2\theta)}\int_M\mathcal{S}w^2 dv_g\\
\le &\frac{s_0(1+\theta)^2}{a (1+2\theta)}\int_M w^2 dv_g.
\end{split}
\ee
Combining this with \cite[Th. 2.3]{LeeParker1987} (take $\e=1$ there), we have
\bee
\begin{split}
||w||_p^2\le& C(n)\int_M(|\nabla w|^2+w^2)dv_g \\
\le& C(n, \theta, s_0)\int_M w^2dv_g\\
=&C(n, \theta, s_0)\int_M w^{2-\e}w^\e dv_g\text{\ \ ($0<\e<2 $ to be chosen)}\\
\le &C(n, \theta, s_0)\lf(\int_M w^p  dv_g\ri)^{\frac{2-\e}p}\lf(\int_M w^{\e\cdot\frac p{p-2+\e} }  dv_g\ri)^{\frac{p-2+\e}p}
\end{split}
\eee
So
$$
\lf(\int_M w^pdv_g\ri)^\frac \e p\le C(n, \theta, s_0) \lf(\int_M w^{\e\cdot\frac p{p-2+\e} }  dv_g\ri)^{\frac{p-2+\e}p}.
$$
Let $\e=\frac{p-2}\theta$ so that
$$
(1+\theta) \e\cdot\frac p{p-2+\e}=p.
$$
If $2\theta>p-2$, then $0<\e<2$, we have
$$
\int_M u^{p(1+\theta)}dv_g\le C(n,s_0,\theta).
$$
This proves the lemma for $q=p(1+\theta)$ with $2\theta>p-2$. If $q=p(1+\theta)$, with $2\theta \le p-2$, the lemma follows from H\"older inequality.
\end{proof}

 \begin{lma}\label{l-u-1} As in Lemma \ref{l-uconvergence-1},
\begin{enumerate}
  \item [(i)] For any $q>p$, there is a constant $C$ independent of $k$ such that
$$
||u_k||_{q,g_0}\le C.
$$
  \item [(ii)]   $u_k$ subconverge in $C^{2}$ norm with respect to $g_0$ in any compact set $K\subset M\setminus \Sigma$.
  \item [(iii)]
$$
\lim_{k\to\infty}\int_M||^{g_0}\nabla u_k||_{g_0}^2dv_{g_0}=0.
$$

\item[(iv)] $\lim_{k\to\infty}\lambda_k=\sigma$.
\end{enumerate}
where $\lambda_k=\lambda_{i_k;\tau_k}$ as in \eqref{e-Yamabe-1}.
\end{lma}
\begin{proof}
Since $\mathcal{S}_{i_k;\tau_k}\ge \sigma-\delta_k$ and $\delta_k\to0$, (i) follows from Lemma \ref{l-Lq-1} and the fact that $C^{-1}g_0\le G_k\le Cg_0$ for some $C>0$ for all $k$.

To prove (ii), for any compact set   $K\subset M\setminus \Sigma$, then there is an open set $ K\Subset U\subset M\setminus \Sigma$ so that $G_k$ converges in $C^\infty$ norm to $g_0$ on $U$. By Lemma \ref{l-Einstein-1}, we conclude that $0\le -\lambda_k\le C$ for some constant independent of $k$.
Then by (i), and   \cite[Th. 2.4]{LeeParker1987}, we conclude that for any $U'\Subset U$,
$$
||u_k||_{L_{2}^q(U')}\le C_1
$$
for some constant $C$ independent of $k$. We then use the Sobolev embedding theorem to conclude that the $C^\a$ norm of $u_k$ are uniformly bounded in $U'\Subset U$. From this the results follows by Schauder estimates.

(iii) and (iv) follows from Lemma \ref{l-Einstein-1}.

\end{proof}

\begin{cor}\label{c-u-1} After passing to a subsequence, $u_k$ converge in $C^2$ norm locally in $M\setminus \Sigma$ to a function $\mathfrak{u}$. Moreover, $\mathfrak{u}=1$  in $M\setminus\Sigma$ and
$$
 \mathcal{S}_{g_0} =\sigma.
$$
In particular  Lemma \ref{l-uconvergence-1} is true.
\end{cor}
\begin{proof} By Lemma \ref{l-u-1}, after passing to a subsequence, $u_k$ converge in $C^2$ norm locally in $M\setminus \Sigma$ to a function $\mathfrak{u}$. Moreover, $\mathfrak{u}$ is constant in each component of $M\setminus\Sigma$. We claim that there is $C_1>0$ such that $0\le u_k\le C_1$ for all $k$.

Since the scalar curvature $\mathcal{S}_{G_k}\ge -s_0$ for some $s_0>0$ independent of $k$ and since $\lambda_k\le0$, we have
$$
-a\Delta_{G_k}u_k-s_0 u_k\le -a\Delta_{G_k}u_k+\mathcal{S}_{G_k}u_k\le 0.
$$
Moreover, $\int_Mu_k^pdv_{G_k}=1$ and $G_k$ is equivalent to $g_0$ uniformly in $k$, the claim follows from mean value inequality \cite[Theorem 8.17]{GT}.

Since $u_k\to \mathfrak{u}$ almost everywhere, and $G_k$ converge uniformly to $g_0$, we have
$$
\int_M \mathfrak{u}^pdv_{g_0}=1.
$$
In particular, $\mathfrak{u}>0$ somewhere.

Next we want to prove that $\mathfrak{u}$ is constant on $M$. By Lemma \ref{l-u-1},  there is a constant $C_2$ independent of $k$ so that
$$
\int_M(|^{g_0}\nabla u_k|^2_{g_0}+u_k^2)dv_{g_0}\le C_2.
$$
Passing to a subsequence, we may assume that $u_k$ converge weakly in $W^{1,2}(M,g_0)$ to $v$ say. We claim that $v$ is constant. In fact, for any $\ell\ge1$, the sequence $u_{\ell+k}$, $k\ge 1$ also weakly converge to $v$. Then we can find convex combinations of $u_{\ell+k}$ which converge  to $v$ strongly in $W^{1,2}(M,g_0)$. Namely, for any $\e>0$, there exists $\a_1,\dots,\a_{m}$ with $\a_k\ge0$, $\sum_{k=1}^m\a_j=1$ such that if $w=\sum_{k=1}^m\a_ku_{\ell+k}$, then
$$
||w-v||_{W^{1,2}(M, g_0})\le\e.
$$
On the other hand, by Lemma \ref{l-u-1}, if  $\ell$ is large enough, then
\bee
\begin{split}
 (\int_M|^{g_0}\nabla w|^2_{g_0}dv_{g_0})^\frac12\le &(\int_M(|\sum_k\a_k|^{g_0}\nabla u_{\ell+k}|_{g_0} )^2dv_{g_0})^\frac12\\
 \le& \sum_k\a_k(\int_M|^{g_0}\nabla u_{k+i}|^2_{g_0}dv_{g_0})^\frac12\\
 \le &\e.
\end{split}
\eee
Hence
$$
\int_M|^{g_0}\nabla v|^2 dv_{g_0}\le (2\e)^2.
$$
This implies $^{g_0}\nabla v=0$, a.e. Since $v\in W^{1,2}(M,g_0)$, we conclude that  $v=c$ is a constant as claimed.

On the other hand, for any smooth function $\phi$ on $M$
\bee
\begin{split}
\lim_{k\to\infty}\int_M(\la ^{g_0}\nabla \phi,  ^{g_0}\nabla u_k\ra_{g_0} +\phi u_i)dv_{g_0}=&\int_M(\la ^{g_0}\nabla \phi,  ^{g_0}\nabla v\ra_{g_0} +\phi v)dv_{g_0}\\
=&\int_M \phi v dv_{g_0}.
\end{split}
\eee
Also by Lemma \ref{l-u-1} again, and the fact that $u_k$ are uniformly bounded and $u_k\to \mathfrak{u}$ a.e., we have

\bee
\lim_{k\to\infty}\int_M(\la ^{g_0}\nabla \phi,  ^{g_0}\nabla u_k\ra_{g_0} +\phi u_i)dv_{g_0}=\int_M \phi \mathfrak{u} dv_{g_0}.
\eee
So
\bee
\int_M \phi \mathfrak{u} dv_{g_0}=\int_M \phi v dv_{g_0}.
\eee
Hence $\mathfrak{u}=v$ is a constant. Since $\int_M\mathfrak{u}^pdv_{g_0}=1 $   so $\mathfrak{u}=1$. Since $\mathfrak{u}$ satisfies:
$$
-a\Delta_{g_0}\mathfrak{u}+\mathcal{S}_{g_0} \mathfrak{u}=\sigma \mathfrak{u}^p,
$$
the last assertion follows.

\end{proof}

This completes the proof of  Theorem \ref{t-Yamabe-1}. Next we want to discuss the case that $\Sigma$ has codimension one. We have the following:
\begin{thm}\label{t-Yamabe-2} Let $M^n$ be a smooth compact manifold such that $\sigma(M)\le 0$. Suppose $g_0$ is a Riemannian metric with $V(M,g_0)=1$ satisfying {\bf   (b1)--(b-3)} in section \ref{s-approx-1}. Then $g_0$ is Einstein on $M\setminus \Sigma$ and $\mathcal{S}_{g_0}=\sigma(M)$.  Moreover, $H_-=H_+$.
\end{thm}
\begin{proof} Let $g_i=g_{\e_i,0}$ be the smooth approximation of $g_0$ by \cite{Miao2002} as given in section \ref{s-approx-1}. The fact that $g_0$ is Einstein outside $\Sigma$ can be proved similarly as above using Corollary \ref{c-approx-2}. It remains to prove that $H_-=H_+$. Let $\e_i\to0$ and let $u_i$ be the positive solution of
$$
-a\Delta_i u_i+\mathcal{S}_iu_i=\lambda_i u_i^{p-1}
$$
normalized as
$$
\int_M u_i^pdv_i=1
$$
Here $\Delta_i$ is the Laplacian of $g_i$ etc. Also $\lambda_i\le \sigma$, where $\sigma:=\sigma(M)$. Suppose $H_-(z)>H_+(z)$ somewhere, then one can easily check that there is a positive constant $b$ such that for $i$ large enough.
\be\label{e-meancurvature-1}
\int_M \mathcal{S}_idv_i\ge \sigma+b.
\ee
As before, passing to a subsequence if necessary, $u_i\to 1$ outside $\Sigma$ and uniform in $C^\infty$ norm in any compact set of $M\setminus \Sigma$. Moreover, $u_i$ are uniformly bounded, and $\lambda_i\to \sigma$. Since $\mathcal{S}_i$ be bounded below by $-s_0$, for some $s_0\ge0$ and $u_i$ is bounded from below, we have
\bee
\begin{split}
\sigma=&\lim_{i\to\infty}\lambda_i\int_Mu_i^{p-1}dv_i \\
=&\lim_{i\to\infty}\int_M \mathcal{S}_iu_idv_i\\
\ge&\lim_{i\to\infty}\int_M\mathcal{S}_i(u_i-1) dv_i+\sigma+b\\
\end{split}
\eee
where we have used the fact that $V(M,g_{0,\e_i})\to V(M,g_0)=1$ and \eqref{e-meancurvature-1}. We claim that
$$
\lim_{i\to\infty}\int_M\mathcal{S}_i(u_i-1) dv_i=0.
$$
If the claim is true, then we have a contradiction because $b>0$. To prove the claim,
note that on $|t|\le a$, the original metric $g_0$ is of the form:

$$
g_0(z,t)=dt^2+g_{ij}(z,t)dz^idz^j.
$$
We assume that $g_{ij}(z,t)$ (which will be denoted by $h^t_{ij }(z)$) is uniformly equivalent to $g_{ij}(z,0)$ (which will be denoted by $h_{ij}(z)$). For any $z\in \Sigma$ and for any $1\ge t\ge 0$,
\bee
|u_i(z,a)-u_i(z,t)|\le \int_0^a\lf|\frac{\p u_i(z,s)}{\p s}\ri|ds\le \int_0^1|^{g_0}\nabla u_i|(z,s)ds
\eee
  By the properties of $g_{0,\e}$,
\be\label{e-scalar-4}
\int_{\frac{\e^2_i}{100}\le |t|\le \e_i}|\mathcal{S}_i(u_i-1)|dv_i=o(1)
\ee
because $u_i$ are uniformly bounded.
\be\label{e-scalar-5}
\begin{split}
&\int_{|t|\le \frac{\e^2_i}{100} } \mathcal{S}_i(z,t)(u_i(z,t)-1) dv_{g_i}\\
=&
\int_{|t|\le \frac{\e^2_i}{100} } \mathcal{S}_i(z,t)(u_i(z,1)-1) dv_{g_i}+
\int_{|t|\le \frac{\e^2_i}{100} } \mathcal{S}_i(z,t)(u_i(z,t)-u_i(z,1)) dv_{g_i}\\
=&I+II.
\end{split}
\ee
Since $u_i(z,1)\to 1$ uniformly on $z\in \Sigma$, and $\int_M |\mathcal{S}_i|dv_{g_i}$ is bounded, we conclude that
\be\label{e-scalar-5}
I=o(1)
\ee
as $i\to\infty$.
On the other hand,
\be\label{e-scalar-6}
\begin{split}
|II|\le &\int_{|t|\le \frac{\e^2_i}{100} } |\mathcal{S}_i(z,t)(u_i(z,t)-u_i(z,1))| dv_{g_i}\\
\le &c\int_{z\in \Sigma}\lf(\int_{-\frac{\e^2_i}{100} }^{ \frac{\e^2_i}{100} }\e_i^{-2}\int_0^1|\nabla u_i(z,s)|ds\ri) dtdv_{h}\\
\le& c\int_{z\in \Sigma}\lf( \int_0^a|\nabla u_i(z,s)|ds\ri) dtdv_{h}\\
\le &c\int_M|\nabla u_i|dv_{g_i}\\
=&o(1)
\end{split}
\ee
by the Schwartz inequality and Lemma \ref{l-u-1}. The claim follows from \eqref{e-scalar-4}--\eqref{e-scalar-6}.
\end{proof}


\section{singular Einstein metrics}\label{s-harmonic-1}

In the conclusions of Theorem  \ref{t-Yamabe-1}, one obtains   metrics which are smooth and Einstein outside some singular sets. In this section, we want to prove that under certain conditions, one may introduce smooth structure so that the Einstein metric is actually smooth. More precisely, we have the following:

\begin{thm}\label{t-harmonic-1}
Let $ M^n $, $n\ge 3$ be a smooth manifold and $g$ is a Riemannian metric on $M$ satisfying the following conditions: There is a compact set $\Sigma$ in $M$ such that
\begin{enumerate}
  \item [(i)] $g$ is Lipschitz and $g$ is smooth on $M\setminus \Sigma$;
  \item [(ii)] $g=\lambda \Ric$ on $M\setminus\Sigma$ for some constant $\lambda$;
  \item [(iii)] codimension of $\Sigma$ is larger than 1 in the sense that
    $V(\Sigma(\e),g)  =O(\e^{1+\theta})$  for some $\theta>0$, where $\Sigma(\e)=\{x\in M|\ d(x,\Sigma)<\e\}$.
\end{enumerate}
Then for any open set $U$ containing $\Sigma$, there is a smooth structure on $M$ which is the same as the original smooth structure on $M\setminus U$ so that $g$ is a smooth Einstein metric on $M$.
\end{thm}

We want to construct the required smooth structure using harmonic coordinates.  First recall the following.

\begin{lma}\label{l-harmonic-1} Let $B(1)$ be the unit ball in $\R^n$ with center at the origin. Let $(a_{ij})$ be a symmetric matrix so that
$$
\lambda|\xi|^2\le a^{ij}\xi^i\xi^j\le \Lambda|\xi|^2,
$$
for some $\Lambda>\lambda>0$ for all $\xi\in \R^n$ and $a^{ij}$ is Lipschitz with Lipschitz constant $L$. Let $f\in L^\infty(B(1))$. Then the following boundary value problem:
\bee
\left\{
  \begin{array}{ll}
   \displaystyle{  \frac{\p}{\p x^i}\lf(a^{ij}\frac{\p u}{\p x^j}\ri)}&=  f \hbox{\ in $B(1)$;} \\
    u&= 0\hbox{\ on $\p B(1)$,}
  \end{array}
\right.
\eee
has a  unique solution in $W^{2,p}(B(1))$ for any  $p>1$  with $u\in W_0^{1,p}(B(1))$.
$$
||u||_{2,p}\le C\lf(||u||_p+ ||f||_p\ri)
$$
for some constant depending only on $p, n, \lambda,\Lambda, L$. Here $||u||_{2,p}$ is the $W^{2,p}$ norm on $B(1)$ and $||u||_p$ is the $L^p$ norm in $B(1)$.
\end{lma}
\begin{proof} The results follow from \cite[Theorem 9.15, Corollary 9.13]{GT}.  By taking $p>n$, by the Sobolev embedding theorem, $u$ is continuous up to the boundary and $u=0$ at the boundary.

\end{proof}
 With the same assumptions and notation as in Theorem \ref{t-harmonic-1}, let $q\in \Sigma$. Let $U_\delta=\{(x^1,\dots,x^n)| |x|<\delta\}$ be smooth local coordinates     neighborhood with   $q$ being at the origin such that $g_{ij}$ is equivalent to the Euclidean metric and $g_{ij}$ is Lipschitz with Lipschitz constant $L$

 \begin{lma}\label{l-harmonic-2} With the above assumptions and notation, there is  $\delta>\e>0$ and   functions $u^1,\dots, u^n$ on $U_\e=\{(x^1,\dots,x^n)| |x|<\e\}$ such that the mapping $(x^1,\dots,x^n)\to (u^1,\dots,u^n)$ is a local $C^{1,\a}$ diffeomorphism at the origin for some $0<\a<1$, $u^i\in W^{2,p}(U_\e)$ for all $p>1$ and $u^i$ is harmonic with respect to $g$ for   $1\le i\le n$. Moreover, $u^i$ is smooth outside $\Sigma$.
 \end{lma}
 \begin{proof} Let $\delta>\e>0$ to be chosen later. Fix $\ell$, let $f=\Delta_g x^\ell$ which is bounded by the assumption on $g_{ij}$. Let $\lambda, \Lambda>0$ be such that

\be\label{e-elliptic}
\lambda|\xi|^2\le g^{ij}\xi^i\xi^j\le \Lambda|\xi|^2,
\ee
in $U_\delta$.

 Let $y=\e^{-1}x$. Consider the following boundary value problem on $B(1)$ in the $y$-space
 \be\label{e-v-1}
\left\{
  \begin{array}{ll}
    \displaystyle{\frac{\p}{\p y^i}\lf(\sqrt gg^{ij}\frac{\p v}{\p y^j}\ri)}&=  \e^2 \sqrt gf \hbox{\ in $B(1)$;} \\
    v&= 0\hbox{\ on $\p B(1)$,}
  \end{array}
\right.
\ee
 By Lemma \ref{l-harmonic-1}, the boundary value problem has a solution $v$ satisfying the conclusions in that lemma. Here we have used the fact that $g_{ij}$ has Lipschitz constant bounded by $\e L$ and still satisfies \eqref{e-elliptic} as functions of $y$. In particular, we have
 $$
 ||v||_{2,p;y}\le C_1\lf(||v||_{p;y}+\e^2\ri).
 $$
 Here and below, $C_i$ will denote positive constants  independent of $\e$. Let $p>n$ be fixed, then one can see that there is $1>\a>0$ such that $v\in C^{1,\a}(B(1))$ in the $y$-space and
 \be\label{e-harmonic-1}
 ||v||_{C^{1,\a}(B(1))}\le   C_2\lf(||v||_{p;y}+\e^2\ri).\\
\ee
 for some positive constants $C_2-C_4$ independent of $\e$.

 Let $w(x)=v(\e^{-1} x )$ with $x\in B(\e)$ in the $x$-space. Then $w$ satisfies
 \bee
\left\{
  \begin{array}{ll}
  \displaystyle{  \frac{\p}{\p x^i}\lf(\sqrt gg^{ij}\frac{\p w}{\p x^j}\ri)}&=   \sqrt gf \hbox{\ in $B(\e)$;} \\
    w&= 0\hbox{\ on $\p B(\e)$,}
  \end{array}
\right.
\eee
in the $x$-space. Moreover, $w\in W^{2,p}(B(\e))$. Let $u^\ell=w-x^\ell$. Then $u^\ell$ is harmonic, namely, $u^\ell$ satisfies
 \bee
\left\{
  \begin{array}{ll}
  \displaystyle{\frac1{\sqrt g}  \frac{\p}{\p x^i}\lf(\sqrt gg^{ij}\frac{\p u^\ell}{\p x^j}\ri)}&= 0  \hbox{\ in $B(\e)$;} \\
   u^\ell&= x^\ell \hbox{\ on $\p B(\e)$,}
  \end{array}
\right.
\eee
By the maximum principle, we conclude that $|u^\ell|\le  \e$ and so $|w|\le 2\e $.
\be\label{e-localdiff-1}
\sup_{B(\e}|\p_x w|=\e^{-1}\sup_{B(1)}|\p_yv|\le  C_2\e^{-1}\lf(||v||_{p;y}+\e^2\ri)
\ee
To estimate the RHS, multiply \eqref{e-v-1} by $v$ and integrating by parts, using the Poincar\'e inequaltiy, we have
\bee
\int_{B(1)}v^2dy\le C_3\e^2\int_{B(1)}|v|dy
\eee
and so
\bee
\begin{split}
||v||_{p;y}\le&\lf(\sup_{B(1)}|v|\ri)^{1-\frac2p}\lf(\int_{B(1)}v^2\ri)^\frac1p\\
\le &C_4\e^{1-\frac2p}\cdot\e^\frac4p\\
=&C_4\e^{1+\frac2p}
\end{split}
\eee
where we have used the H\"older inequality and the fact that $|v|=|w|\le 2\e$. By \eqref{e-localdiff-1} we conclude that
\bee
\sup_{B(\e)}|\p_x w|\le C_5\e^{\frac2p}.
\eee
Hence
\bee
\frac{\p u^\ell}{\p x^i}=\delta_i^\ell+O(\e^{\frac2p}).
\eee
From this and the fact that $g$ is smooth outside $\Sigma$ it is easy to see that the lemma is true, provided $\e$ is small enough.
 \end{proof}

 \begin{proof}[Proof of Theorem \ref{t-harmonic-1}] Let $U$ be any open set containing $\Sigma$. For any $q\in \Sigma$, by Lemma \ref{l-harmonic-2}, we can find smooth coordinates neighborhood $V_q\Subset U$ around $q$ and $C^{1,\a}$ functions $u^1,\dots,u^n$ on   $V_q$  near $q$ which are in  $  W^{2,p}(V_q)$ as functions of $x$. Moreover, $(x^1,\dots,x^n)\to (u^1,\dots,u^n)$ is a $C^1$ diffeomorphism from $V_q$ to its image $\wt V_q$ in  the $u$-space. Let
 \be\label{e-h-1}h_{ab}=g(\frac{\p}{\p u^a},\frac{\p}{\p u^b})=\frac{\p x^i}{\p u^a}\frac{\p x^j}{\p u^b}g_{ij},
\ee
 where $g_{ij}=g(\frac{\p}{\p x^i},\frac{\p}{\p x^j})$.  Let $R_{ab}=\Ric(\frac{\p}{\p u^a},\frac{\p}{\p u^b})$.  Since each $u^a$  is harmonic, and $R_{ab}=\lambda h_{ab}$ by assumption, away from $\Sigma$ for all $a, b$ we have
 \be\label{e-Ricci-harmonic-1}
 h^{cd}h_{ab,cd}=-2\lambda h_{ab}+\p h^{-1}*\p h+ h^{-1}*h^{-1}*\p h*\p h:=Q(h,\p h).
 \ee
 where $(h^{cd})=(h_{cd})^{-1}$, $h_{ab,c}=\frac{\p}{\p u^c}h_{ab}$  etc, and $\p h^{-1}*\p h$ denotes a sum of finite terms of the form $(\frac{\p}{\p u^c}h^{ab})(\frac{\p}{\p u^f}h_{de})$ etc. By \eqref{e-h-1},
 \be\label{e-h-2}
 h_{ab,c}=2\frac{\p^2x^i}{\p u^a\p u^c}\frac{\p x^j}{\p u^b}g_{ij}+\frac{\p x^i}{\p u^a }\frac{\p x^j}{\p u^b}\frac{\p x^k}{\p u^c}\frac{\p }{\p x^k}g_{ij}.
 \ee
We may assume that $\wt V_q$ contains the origin which is the coordinates of $q$. Then by shrinking $\wt V_q$ is necessary, by Lemma \ref{l-harmonic-2}, $h_{ab}$ is bounded  and $h_{ab,c}$ is in $L^p$ for all $p>1$ for all $a, b, c$ as functions of $u$. In particular, $h_{ab}$ is in $W^{1,p}(\wt V_q)$ for all $p>1$. Moreover, $(h^{ab})$ is uniformly elliptic. Since  $h^{ab}$ is only in $C^\a$ with $0<\a<1$, we cannot apply standard $L^p$ estimate as in \cite[Theorem 9.19]{GT}. Hence, we want to prove that $h_{ab}$ is in $W^{2,p}(B(\delta))$ for all $a, b$ for all $p>n$ and for some $\delta>0$ in the $u$-space, where $B(\delta)=\{u|\ |u|<\delta\}$. Suppose this is true, then $h_{ab}\in C^{0,1}_{{\rm loc}}(B(\delta))$ and $\p h\in W^{1,p}_{{\rm loc}}(B(\delta))$. This implies $Q(h,\p h)$ in \eqref{e-Ricci-harmonic-1} is in $W^{1,\frac p2}_{{\rm loc}}(B(\delta))$. Since this is true for all $p>n$, by \cite[Theorem 9.19]{GT}, we conclude that $h_{ab}$ is in $W^{3,p}(B(\delta))$. Continue in this way, we conclude that $h_{ab}\in W^{k,p}_{{\rm loc}}(B(\delta))$ for all $k\ge 1$ and $p>n$ by booth trap argument. Hence $h_{ab}$ is smooth near the origin.

It remains to prove that $h_{ab}\in W^{2,p}(B(\delta))$ for all $p>n$ for all $a, b$ for some $\delta>0$. Fix $a, b$ and let $w=\phi h_{ab}$ where $\phi$ is  a smooth cutoff function in $B(2\delta)$ so that $\phi=1$ in $B(\delta)$, $\phi=0$ outside $B(\frac32\delta)$, where $\delta>0$ is small enough so that $B(2\delta)\Subset \wt V_q$. Then away from $\Sigma$, $w$ satisfies:

\be\label{e-Ricci-harmonic-2}
 h^{cd}w_{cd}=Q_1(h,\p h, \phi, \p \phi, \p^2\phi).
 \ee
Since $Q_1$ is in $L^p(B(2\delta))$ by Lemma \ref{l-harmonic-2} and $(h^{cd})$ is continuous and is uniformly elliptic, by \cite[Theorem 9.15]{GT} for any $p<n$ there is $v\in W^{2,p}(B(2\delta))\cap W_0^{1,p}(B(2\delta))$ such that
\bee
 h^{cd}v_{cd}=Q_1(h,\p h, \phi, \p \phi, \p^2\phi).
\eee
Since $h^{cd}\in W^{1,p}(B(2\delta))$ for all $p$, for any smooth function $\eta$ with compact support in $B(2\delta)$, we have
\be\label{e-Ricci-harmonic-3}
\int_{B(2\delta)}\lf(h^{cd}\frac{\p\eta}{\p u^a} \frac{\p v}{\p u^b}+\eta s^d\frac{\p v}{\p u^d}\ri)du=-\int_{B(2\delta))}\eta Q_1 du.
\ee
where $s^d=\frac{\p}{\p u^c}h^{cd}$. We want to prove that $w$ also satisfies this relation.

To prove the claim, note that if we consider $\Sigma\cap \wt V_q$ then the codimension of $\Sigma$ in the $u$-space is at least $1+\theta$ for some $\theta>0$ because $h_{ab}$ and the Euclidean metric are uniformly equivalent. As in \cite{Lee2013}, for $\e>0$ small enough, we can find a smooth function $0\le \xi_\e\le 1$ in $\wt V_q$ such that $\xi_\e=1$ outside $\Sigma_{2\e} $ and is zero in $\Sigma_\e\cap \wt V_q$ where $\Sigma_\e=\{u\in \wt V_q| d(u,\Sigma)<\e\}$ where the distance is the Euclidean distance. Moreover, $|\p \xi_\e|\le C_1\e^{-1}$. Here and below $C_i$ denotes a positive constant independent of $\e$. Now let $\eta$ be a smooth function with compact support in $B(2\delta)$. Multiply \eqref{e-Ricci-harmonic-2} by $\eta\xi_\e$ and integrating by parts, we have

\bee
\begin{split}
-\int_{B(2\delta)}\eta \xi_\e Q_1du=& \int_{B(2\delta)}\lf[h^{cd}\lf(\xi_\e\frac{\p \eta}{\p u^a} + \eta\frac{\p \xi_\e}{\p u^a}\ri) \frac{\p w}{\p u^b}+\eta \xi_\e s^d\frac{\p w}{\p u^d}\ri]du
\end{split}
\eee
Since $w, h^{cd}\in L^{1,p}(B(2\delta))$ for all $p>1$, we have

$$
\int_{B(2\delta)}|\eta (\xi_\e-1) Q_1|du\le \lf(\int_{B(2\delta)} |\eta (\xi_\e-1) Q_1|^2\ri)^\frac12 V(\Sigma_{2\e})^\frac12\to 0
$$
as $\e\to0$. Similarly, one can prove that

$$
\int_{B(2\delta)}\lf|h^{cd} (\xi_\e-1)\frac{\p \eta}{\p u^a}   \frac{\p w}{\p u^b} + \eta (\xi_\e-1) s^d\frac{\p w}{\p u^d}\ri| du\to0
$$
as $\e\to0$. On the other hand,
\bee
\begin{split}
 \int_{B(2\delta)}\lf|h^{cd}  \eta\frac{\p \xi_\e}{\p u^a} \frac{\p w}{\p u^b}\ri|du
 \le&C_2\e^{-1}\int_{\Sigma_{2\e}}|\p w|du\\
 \le &C_3 \e^{-1}\lf(\int_{\Sigma_{2\e}}|\p w|^pdu\ri)^\frac1p (V(\Sigma(2\e)))^(1-1p)\\
\le &C_4\e^{-1+(1+\theta)(1-\frac1p)}\lf(\int_{\Sigma_{2\e}}|\p w|^pdu\ri)^\frac1p\\
\to&0
\end{split}
\eee
as $\e\to0$ provided $p$ is large enough.  Hence we have
\be\label{e-Ricci-harmonic-4}
\int_{B(2\delta)}\lf(h^{cd}\frac{\p\eta}{\p u^a} \frac{\p w}{\p u^b}+\eta s^d\frac{\p w}{\p u^d}\ri)du=-\int_{B(2\delta))}\eta Q_1 du.
\ee
for all smooth function $\eta$ with compact support $B(2\delta)$.

Let $\zeta=v-w$, then $v-w\in W_0^{1,p}$ for all $p>1$ and
\be\label{e-Ricci-harmonic-5}
\int_{B(2\delta)}\lf(h^{cd}\frac{\p\eta}{\p u^a} \frac{\p \zeta}{\p u^b}+\eta s^d\frac{\p \zeta}{\p u^d}\ri)du=0
\ee
 for all smooth function $\eta$ with compact support in $B(2\delta)$. Using the fact that $s^d\in L^p(B(2\delta))$ we can proceed as in the proof of \cite[Theorem 8.1]{GT} to conclude that $\zeta\equiv0$. Since it is assumed that $s^d$ is bounded in that theorem  and we only have $s^q\in L^p(B(2\delta))$ for all $p>1$ in our case, we sketch the proof as follows. Suppose $\sup_{B(2\delta)}\zeta=m>0$ which is finite because $\zeta$ is continuous. For any $m>\tau>0$, let $\zeta_\tau=\max\{\zeta-\tau,0\}$. Multiply \eqref{e-Ricci-harmonic-5} by $\zeta_\tau$ and integrating by parts, using the uniform ellipticity of $h^{cd}$ and the Sobolev inequality, we have
 \bee
 \begin{split}
 ||\zeta_\tau||_{\frac{2n}{n-2}}\le& D_1||\p \zeta_\tau||_2\\
 \le &D_2\lf(\int_{\Gamma_\tau} \zeta_\tau^2 \mathbf{s}^2du\ri)^\frac12\\
 \le &D_2||\zeta_\tau||_{\frac{2n}{n-2}}\lf(\int_{\Gamma_\tau}\mathbf{s}^{n}\ri)^\frac1n\\
 \le &D_3||\zeta_\tau||_{\frac{2n}{n-2}}|\Gamma_\tau|^\frac1{2n}
 \end{split}
 \eee
here and below $D_i$ denotes a positive constant independent of $\tau$, $\mathbf{s}=\sqrt{\sum_{d}(s^d)^2}$ and $|\Gamma_\tau|$ is the measure of  the support of $\p \zeta_\tau$. Hence
$$
|\Gamma_\tau|^\frac1{2n}\ge D_3^{-1}
$$
for all $m>\tau>0$.
Since support of $\p \zeta_\tau$ is a subset of the support of $\zeta_\tau$, and $\cap_{\tau}\text{supp}(\zeta_\tau)=\{\zeta=m\}$, we have
$$\bigcap_\tau (\Gamma_\tau\cap \{\zeta=m\})$$ has positive measure. But for almost all $x\in \{\zeta=m\}$, $\p\zeta(x)=0$. Hence for almost all $x$ in $ \Gamma_\tau\cap \{\zeta=m\}$, $\p \zeta_\tau(x)=0$. So $|\Gamma_\tau\cap \{\zeta=m\}|=0$. This is impossible.

To summarize we have proved that $h_{ab}\in W^{2,p}(B(2\delta))$  for all $p>n$ and $h_{ab}$ is smooth in $u$ for all $a, b$.

We can cover $\Sigma$ by such harmonic coordinate neighborhoods $V_q$ so that the components of $g$ are smooth with respect to these coordinates. By \cite[Theorem 2.1]{Taylor2006} one can conclude that the theorem is true.

 \end{proof}

  \begin{cor}\label{c-Einstein-1} Suppose $(M^n,g_0)$ is as in Theorem  \ref{t-Yamabe-1}. If in addition,  $g_0$ is Lipschitz. Then there is a smooth structure on $M$ so that $g_0$ is smooth and Einstein.

 \end{cor}
\section{a positive mass theorem with singular set}\label{s-pmt}

In this section, we will use the results in sections \ref{s-gradient} and \ref{s-approx-1} to study positive mass theorems on asymptotically flat manifolds with singular metrics. We want to discuss the theorem without assuming that the manifold is spin. There are different definitions for asymptotically flat manifold. For our purpose, we use the following:

\begin{definition}\label{defaf}
An $n$ dimensional Riemannian  manifold $(M^n,g)$, where $g$ is continuous,  is said to be asymptotically flat
(AF)  if there is a compact subset $K$ such that $g$ is smooth on  $M\setminus K$, and $M\setminus K$ has finitely many components $E_k$, $1\le k\le l$,  each $E_k$ is called an end of $M$, such that each $E_k$ is diffeomorphic to $\R^n\setminus B(R_k)$ for some Euclidean ball $B(R_k)$, and  the following are true:
\begin{enumerate}
  \item [(i)] In the standard coordinates $x^i$ of $\R^n$,
  \bee
g_{ij}=\delta_{ij}+\sigma_{ij}
\eee
with
\bee \label{daf2}
\sup_{E_k} \lf\{\sum_{s=0}^2|x|^{\tau+s}|\p^s\sigma_{ij}| +[|x|^{\a+2+\tau}\p\p \sigma_{ij}]_\a \ri\}<\infty,
\eee
for some $0<\a\le 1$, $\tau>\frac{n-2}2$, where $\p f$ and $\p^2f$ are the gradient and Hessian of $f$ with respect to the Euclidean metric, and $[f]_\a$ the $\a$-H\"older norm of $f$.
  \item [(ii)] The scalar curvature $\mathcal{S}$ satisfies the decay condition:
  $$
  |\mathcal{S}|(x)\le C(1+d(x))^{-q}
  $$
  for some $n+2\ge q>n$. Here $d(x)$ is the distance function from a fixed point in $M$.

\end{enumerate}
The coordinate chart satisfying (i) is said to be {\it admissible}.
\end{definition}

  In the following, for a function $f$ defined near infinity or $\R^n$, and for $k\geq 0$,  $f=O_k(r^{-\tau})$ refers to $\sum^k_{i=0} r^i |\p^i f|=O(r^{-\tau})$ as $r\to\infty$, where $r=|x|$.

\begin{definition}
The Arnowitt-Deser-Misner (ADM) mass (see \cite{ADM}) of an end $E$ of an AF  manifold $M$ is defined as:
\begin{equation} \label{defadm1}
\m_{ADM}(E)=\lim_{r\to\infty}\frac{1}{4(n-1)\omega_{n-1}}\int_{S_r}
\lf(g_{ij,i}-g_{ii,j}\ri)\nu^jd\Sigma_r^0,
\end{equation}
in an admissible coordinate chart where $S_r$ is the Euclidean sphere, $\omega_{n-1}$ is the volume of $n-1$ dimensional unit sphere,  $d\Sigma_r^0$ is the volume
element induced by the Euclidean metric, $\nu$ is the outward unit
normal of $S_r$ in $\R^n$ and the derivative is the ordinary partial
derivative.
\end{definition}
By the result
of Bartnik \cite{BTK86}, the $\m_{ADM}(E)$ is well-defined, i.e. it is independent of the choice of admissible charts.

For smooth metrics, without assuming the manifold is spin, we have the following positive mass theorem by Schoen and Yau \cite{SY1979,SY1981,Schoen1989}:
\begin{thm}\label{t-SY}
Let $(M^n,g)$, $3\le n\le 7$, be an AF manifold with nonnegative scalar curvature $\mathcal{S}\ge0$. Then the ADM mass of each end is nonnegative. Moreover, if the ADM mass of one of the ends is zero, then $(M^n,g)$ is isometric to $\R^n$ with the standard metric.

\end{thm}

We want to prove the following positive mass theorem   for metrics which are smooth outside  a compact set of codimension at least 2. More precisely, we want to prove the following:
\begin{thm}\label{t-pmtsing}
Let $(M^n,g_0)$ be an AF  manifold  with $3\leq n\leq 7$,  $g_0$ being a continuous metric on $M$ such that
\begin{enumerate}
  \item [(i)] $g_0$ is smooth outside a compact set $\Sigma$ with codimension at least 2 as in {\bf (a4)} in section \ref{s-approx-1}.
  \item [(ii)] The scalar curvature $\mathcal{S}$ of $g_0$ is nonnegative outside $\Sigma$.
  \item [(iii)]  $g_0\in W^{1,p}_{\text{loc}}$ for some $p>n$ as in {\bf (a2)} in section \ref{s-approx-1}.
  \item[(iv)] On each end $E$, in an admissible coordinate chart,
  $$
  g_{ij}=\delta_{ij}+\sigma_{ij}
  $$
  with $\sigma_{ij}=O_5 (r^{-\tau})$  with $\tau>\frac{n-2}2$.
\end{enumerate}
Then the ADM mass of each end is nonnegative. Moreover, if the mass of one of the ends is zero, then $M$ is diffeomorphic to $\R^n$, and $g_0$ is flat outside $\Sigma$.

\end{thm}
\begin{remark}
\begin{enumerate}
  \item [(a)] The assumption of continuity of metric  cannot be removed. See the construction  in Proposition \ref{p-cone-pm-2}.
  \item [(b)] The case that the singular set is an embedded hypersurface has been studied  in \cite{Miao2002,ShiTam2002},  see also  \cite{McFeronSzekelyhidi2012}.
  \item [(c)] In case the singular set has codimension larger than 1, for spin manifolds, positive mass theorems  have been obtained under rather general assumptions in \cite{LeeFeloch2015}. Without the spin condition, there are  also results for metrics with bounded $C^2$ norm and with singular set to have codimension at least $n/2$ \cite{Lee2013}.
\end{enumerate}

\end{remark}

We proceed as in \cite{McFeronSzekelyhidi2012}. As in section \ref{s-approx-1}, let $\e>0$, $\e\to0$, we can construct a family of metrics $g_{\e,0}$ such that
\begin{enumerate}
  \item [(i)] $g_{\e,0}\to g_0$ uniformly.
  \item [(ii)] $g_{\e,0}=g_0$ outside $\Sigma(2\e)$.
  \item [(iii)] The $W^{1,p}$ norm of $g_{\e,0}$ in a fixed precompact open set   containing $\Sigma$ is bounded by a constant independent of $\e$.
\end{enumerate}

As in section \ref{s-approx-1}, we can choose $\e_0>0$ small enough and let $h=g_{\e_0,0}$. Then there is a $T>0$ independent of $\e$ such that if $0<\e \le \e_0$, then there is a smooth solution $g_\e(t)$ on $M\times[0,T]$ to the $h$-flow with initial data $g_{\e,0}$. There is also a smooth solution $g(t)$ on $M\times(0,T]$ to the $h$-flow such that $g(t)\to g_0$ uniformly on compact sets as $t\to0$. Moreover, Lemma \ref{l-approx-2} is still true with $M$ being noncompact in this case because $M$ is AF.

Let $\wt g_\e(t)$ be the corresponding solution to the Ricci flow with $\wt g_\e(t)=\Phi_t^*(g_\e(t))$ as in the compact case in section \ref{s-approx-1}. Then we have the following:

\begin{lma}\label{l-approx-1}
\begin{enumerate}
  \item [(i)] $g_\e(t)$, $\wt g_\e(t)$,  $g(t)$ are AF in the sense of Definition \ref{defaf}.
  \item [(ii)] For each end $E$ of $M$, $\m(E)(\e,t)=\m(E)(\e,0) =\m(E)$ where $\m(E)(\e,t)$ is the mass with respect to $g_\e(t)$ or $ \wt g_\e(t)$;   and $\m(E)$ are the masses with respect to $g_{\e,0}$ or $g_0$.
\end{enumerate}

\end{lma}
\begin{proof}
(i) First note that $C_1^{-1}h\le g_\e(t)\le C_1h$ for some $C_1>0$ independent of $\e, t$. On the other hand, by Lemma \ref{l-approx-2} applied to the noncompact case, we conclude that the curvature of $\wt g_\e(t)$ is bounded by $Ct^{-\frac12(1+\delta)}$ for some $0<\delta<1$ where $C, \delta $ are independent of $\e, t$. Hence we also have
$C_1^{-1}g_{\e,0}\le g_\e(t)\le C_1g_{\e,0}$ and $C_1^{-1}h\le \wt g_\e(t)\le C_1h$, with possible larger $C_1$.

 Using the fact that $\sigma_{ij}=O_5(r^{-\tau})$, we can proceed with some modifications as in \cite{DaiMa2007,McFeronSzekelyhidi2012} to show that outside a fixed compact set, for $0\le l\le 3$,

$$
|^h\nabla^lg_\e(x,t)|\le C_2d^{-l-\tau}(x)
$$
for some constant $C_2$ independent of $\e,t,x$, where $d(x)$ is the distance function from a fixed point with respect to $h$. Here we use the fact that $\sigma_{ij}=O_5(r^{-\tau})$. The proof is similar to the proof for the decay rate of scalar curvature. So we only carry out the proof for this case in more details.

We want to prove the following: There is a constant $C_3>0$ independent of $\e, t$ and a compact set $K$ such that if $\wt{\mathcal{S}}_\e(t)$ is the scalar curvature of $\wt g_\e(t)$, then
\be\label{e-scalar-1}
\sup_{M\setminus K}d^q(x)|\wt{\mathcal{S}}_\e(x,t)|\le C_3.
\ee

We will prove this on each end. Fix $\e$. Denote the scalar curvature of $g_\e(t)$ simply by $\mathcal{S}$ and curvature by $\Rm$ etc.  Let $E$ be an end which is diffeomorphic to $\R^n\setminus B(R)$, say. By the result of \cite{Simon2002}, by choosing $R$ large enough, so that $g_{\e,0}=h=g_0$ outside $B(\frac R2)$ and $g_0$ is smooth there,  we may assume that $|\Rm(g_\e(t))|\le C_4$ for some constant $C_4$ independent of $\e, t$ outside $B(\frac R2)$. Here we have used the fact that $g_\e(t), \wt g_\e(t)$ are uniformly equivalent.

Let $g_e$ be the standard Euclidean metric and let $0\le\phi\le 1$ be a fixed smooth function on $\R^n$ so that $\phi=1$ in $B(R)$ and $\phi=0$ outside $B(2R)$. Consider the metric $\phi g_e+(1-\phi)g_\e(t)$. Still denote its curvature by $\Rm$ etc.

Let $\rho$ be a fixed function $\rho\ge 1$, $\rho=1$ in $B(R)$, $\rho(x)=|x|$ outside $B(2R)$. Hence the gradient and the Hessian of $\rho$ with respect to $g_\e(t)$ are bounded by a constant independent of $\e, t$.

$$
\frac{\p}{\p t}\mathcal{S}^2\le \Delta \mathcal{S}^2+C_5,
$$
in $B(2R)$ and

$$
\frac{\p}{\p t}\mathcal{S}^2=\Delta \mathcal{S}^2+2\mathcal{S}|\Ric|^2-2|\nabla \mathcal{S}|^2,
$$
outside $B(2R)$,

Let  $F=\rho^{ 2q}\mathcal{S}^2$, then outside $B(2R)$.
\bee
\begin{split}
\lf(\frac{\p}{\p t}-\Delta\ri)F= &\rho^{2q}\lf(2\mathcal{S}|\Ric|^2-2|\nabla \mathcal{S}|^2\ri)-2\la\nabla \rho^{2q},\nabla \mathcal{S}^2\ra\\
\le &C_6\rho^{q-4-2\tau}\rho^q\mathcal{S}-4q\rho^{-1}\la\nabla\rho,\nabla F\ra+C_6F\\
\le &C_7-4q\rho^{-1}\la\nabla\rho,\nabla F\ra+C_7F
\end{split}
\eee
for some constants $C_6, C_7$ independent of $\e, t$. The inequality is still true in $B(2R)$ because in $B(R)$, $\nabla \rho=0$ and in $B(2R)\setminus B(R)$, $|\nabla\rho|$ and $|\nabla\mathcal{S}|  $ are uniformly bounded. Hence if $\wt F=e^{-C_7t}F-C_7t$
\be\label{e-scalar-2}
\lf(\frac{\p}{\p t}-\Delta\ri)\wt F\le -4q\rho^{-1}\la\nabla\rho,\nabla \wt F\ra
\ee
Let $A>0$ to be chosen later. Let $\eta=\exp(2At+\rho)$. Then
\bee
\lf(\frac{\p}{\p t}-\Delta\ri)\eta\ge 2A\eta-C\eta
\eee
for some constant $C$ independent of $\e, t$. Choose $A=C$, then we have
\bee
\lf(\frac{\p}{\p t}-\Delta\ri)\eta\ge A\eta.
\eee
 Let $\kappa>0$ be any positive number, then
 $$
 \lf(\frac{\p}{\p t}-\Delta\ri)(\wt F-\kappa \eta)\le  -4q\rho^{-1}\la\nabla\rho,\nabla \wt F\ra-\kappa A\eta.
 $$
Since $\wt F$ is at most polynomial growth, if $\wt F-\kappa\eta$ has a positive maximum, then at some point $(x_0,t_0)$.  Suppose  $t_0>0$, then at $(x_0,t_0)$,
$$
\nabla \wt F=\kappa \nabla \eta.
$$
Hence at $(x_0,t_0)$
\bee
\begin{split}
0\le &\lf(\frac{\p}{\p t}-\Delta\ri)(\wt F-\kappa \eta)\\
\le&  -4q\rho^{-1}\la\nabla\rho,\nabla \wt F\ra-\kappa A\eta\\
=&   -4q\rho^{-1}\kappa\la\nabla\rho,\nabla \eta \ra-\kappa A\eta\\
\le& -\kappa A\eta
\end{split}
\eee
which is impossible. Hence either $\wt F-\kappa \eta\le 0$, or
$$
\wt F-\kappa\eta\le \sup_{\R^n}\lf(\rho^{2q}(x) \mathcal{S}^2(0)\ri)
$$
 where $\mathcal{S}(0)$ is the scalar curvature of $\phi g_e+(1-\phi g_0)$. Let $\kappa\to0$, we conclude the \eqref{e-scalar-1} is true.

(ii) Since $g_{\e,0}=g_0$ outside a compact set, $\m(E)=\m(E)(\e,0)$. On the other hand by the fact that $\wt g_\e(t)$ and $\wt g(t)$ are given by a diffeomorphism and by (i) and \cite{BTK86}, the mass of $E$ is the same whether it is computed with respect to $\wt g_\e(t)$ or $g_\e(t)$.

The fact that $\m(E)(\e,t)=\m(E)(\e,0)$ follows from \cite{DaiMa2007}.

\end{proof}

\begin{proof}[Proof of Theorem \ref{t-pmtsing}] By Lemmas \ref{l-approx-1} and \ref{l-noncompact-1}, we conclude that $g(t)$ is AF and with nonnegative scalar curvature for $t>0$. Let $E$ be an end, using the notation as in   Lemma \ref{l-approx-1}, by the lemma and \cite[Theorem 14]{McFeronSzekelyhidi2012}, the mass $\m(E)(t)$ of $E$ with respect to $g(t)$ satisfies,
\bee
\begin{split}
\m(E)=&\liminf_{\e\to 0}\m(E)(\e,0)\\
=&\liminf_{\e\to 0}\m(E)(\e,t)\\
\ge& \m(E)(t).
\end{split}
\eee
By Theorem \ref{t-SY}, $\m(E)(t)\ge0$, we have $\m(E)\ge0$. If $\m(E)=0$, then $\m(E)(t)=0$ and $(M^n,g(t))$ is isometric to the Euclidean space. Since $g(t)$ converges to $g_0$ in $C^\infty$ as $t\to 0$ away from $\Sigma$, $g_0$ is flat outside $\Sigma$.
\end{proof}

\end{document}